\documentclass[a4paper]{easychair}
\usepackage{amsmath,amssymb,amscd}
\usepackage{graphicx}
\usepackage{tikz}

\newtheorem{theorem}{Theorem}[section]
\newtheorem{corollary}{Corollary}[section]
\newtheorem{lemma}{Lemma}[section]
\newtheorem{proposition}{Proposition}[section]
\newtheorem{definition}{Definition}[section]
\newtheorem{remark}{Remark}[section]

\newcommand{\Go}[1]{{G\"{o}del} }
\newcommand{\termDef}[1]{{\textbf{\textit{#1}}}}

\usepackage{color}

\newcommand{\blue}[1]{{\color{blue} #1}}

\newenvironment{proof2}{\trivlist\item[\hskip
       \labelsep{\it Proof:\/}]\ignorespaces }{\hfill$\blacksquare $\endtrivlist}

\newcommand{\la}{\langle}
\newcommand{\ra}{\rangle}
\renewcommand{\to}{\mathop{\rightarrow}} 

\newcommand{\f}{\ensuremath{\varphi}}
\newcommand{\p}{\ensuremath{\psi}}

\newcommand{\bo}{\ensuremath{\Box}}
\newcommand{\di}{\ensuremath{\Diamond}}
\newcommand{\frm}[1]{\mathfrak{#1}}
\newcommand{\m}{\ensuremath{\frm{M}}}
\newcommand{\wh}[1]{\widehat{#1}}
\newcommand{\De}{\mathrm{\Delta}}

\newcommand{\Si}{\mathrm{\Sigma}}

\newcommand{\mdl}[1]{\models_{\lgc{#1}}}

\newcommand{\mfml}{\ensuremath{\mathcal{L}_{\Box \Diamond }(Var)}} 

\newcommand{\lgc}[1]{\mathsf{#1}}
\newcommand{\alg}[1]{\mathbf{#1}}

\newcommand{\gtop}{\top}
\newcommand{\gbot}{\bot}

\begin{document}

\date{}
\title{Simplified Kripke semantics for K45-like G\"odel modal logics and its axiomatic extensions}
\author{Ricardo Oscar Rodriguez\inst{1} \and  Olim Frits Tuyt \inst{2} \and Francesc Esteva\inst{3} \and Llu\'{\i}s Godo\inst{3}} 

\institute{Departamento de Computaci\'on, FCEyN - UBA, Argentina \\
\email{ricardo@dc.uba.ar}\\
\and
Mathematical Institute, University of Bern, Switzerland \\
\email{olim.tuyt@math.unibe.ch} 
\and 
Artificial Intelligence Research Institute, IIIA - CSIC, Bellaterra, Spain \ \email{\{esteva,godo\}@iiia.csic.es}
}

\authorrunning{} 
\titlerunning{} 

\maketitle


\section{Introduction}\label{intro}
{\em Possibilistic logic} \cite{DuLaPr,DuPr2} is a well-known uncertainty logic  to reason with graded (epistemic) beliefs on classical propositions by means of necessity and possiblity measures.
In this setting, epistemic states of an agent are represented by possibility distributions. If $W$ is a set of classical evaluations or possible worlds, for a given propositional language, a {\em normalized possibility distribution} on $W$ is a mapping  $\pi: W  \to [0, 1]$, with $\sup_{w \in W} \pi(w) = 1$. This map $\pi$ ranks interpretations according to its plausibility level: $\pi(w) = 0$ means that $w$ is rejected, $\pi(w) = 1$ means that $w$ is fully plausible, while $\pi(w) < \pi(w')$ means that $w'$ is more plausible than $w$.
 A possibility distribution $\pi$ induces a pair of dual possibility and necessity measures on propositions, defined respectively as:

\begin{center}
$\Pi(\varphi) = \sup\{ \pi(w) \mid w \in W , w(\varphi) = 1 \}$ \\
$N(\varphi) = \inf\{ 1- \pi(w) \mid w \in W , w(\varphi) = 0\}$
\end{center}

$N(\varphi)$ measures to what degree $\varphi$ can be considered certain given the given epistemic, while $\Pi(\varphi)$ measures the degree in which $\varphi$ is plausible or possible.  Both measures are dual in the sense that $\Pi(\varphi) = 1 - N(\neg \varphi)$, so that the degree of possibility of a proposition $\varphi$ equates the degree in which $\neg \varphi$ is not certain.
If the normalized condition over possibility distribution is dropped, then we gain the ability to deal with inconsistency. In \cite{DyP2015}, a possibility distribution which satisfies $\sup_{w \in W } \pi(w) < 1$ is called sub-normal. In this case, given a set $W$ of classical interpretations,  a degree of inconsistency can be defined in the following way:
$$ inc(W) = 1 - \sup_{w \in W } \pi(w)$$
When the normalised possibility distribution $\pi$ is $\{0, 1\}$-valued, i.e. when $\pi$ is the characteristic function of a subset $\emptyset \neq E \subseteq W$, then the structure $\langle W, \pi \rangle$, or better $\langle W, E \rangle$, can be seen in fact as a K45 frame. Indeed, it is folklore that modal logic K45, which is sound and complete w.r.t. the class of Kripke frames $\langle W, R \rangle$ where $R$ is a euclidean and transitive binary relation, also has a simplified semantics given by the subclass of frames $\langle W, E 
\rangle$, where $E$ is 
a non-empty subset of $W$ (understanding $E$ as its corresponding binary relation $R_E$ defined as $R_E(w, w')$ iff $w' \in E$).

When we go beyond the classical framework of Boolean algebras of events to many-valued frameworks,  one has to come up with appropriate extensions 
of the notion of necessity and possibility measures for many-valued events \cite{DeGoMa}.

In the particular context of G\"odel fuzzy logic \cite{Hajek98}, natural generalizations that we will consider in this paper are the following.
A possibility distribution is as before a function $\pi \colon W \to [0, 1]$ but now,  $W$ is a set of G\"odel propositional interpretations that induces the 
following generalized possibility and necessity measures over G\"odel logic propositions:
$$\Pi(\varphi) = \sup_{w \in W} \{ \min(\pi(w),  w(\varphi)) \}$$ \vspace{-4mm}
$$N( \varphi) = \inf_{w \in W}  \{\pi(w) \Rightarrow w(\varphi)  \},$$
where $\Rightarrow$ is G\"odel implication, that is, for each $x, y \in [0, 1]$, $x \Rightarrow y = 1$ if $x \leq y$,  $x \Rightarrow y = y$, otherwise.\footnote{Strictly speaking, the possibility measure is indeed a generalization of the classical one, but the necessity measure is not, since $x \Rightarrow 0 \neq 1-x$.}
These expressions agree with the ones commonly used in many-valued modal Kripke frames $\langle W, R \rangle$ to respectively evaluate modal formulas $\Diamond \varphi$ and $\Box \varphi$ (see for example \cite{BouJLC} and references 
therein) when the $[0, 1]$-valued accessibility relation $R: W \times W \to [0, 1]$
is  replaced by a possibility distribution $\pi: W \to [0, 1]$ as $R(w,w') = \pi(w')$, for any $w, w' \in W$.

Actually, modal extensions of G\"odel fuzzy logic have been studied by Caicedo and Rodr\'iguez \cite{CaiRod2015}, providing sound and complete axiomatizations for different classes of
$[0, 1]$-valued Kripke models. These structures are triples $\m = \langle W, R, 
e \rangle$, where $W$ is a set of worlds, $R $ is a $[0, 1]$-valued accessibility relation, as above, and $e: W \times Var \to [0, 1]$ is such that, for every $w \in W$, $e(w, \cdot)$  is a G\"odel $[0,1]$-valued evaluation of 
propositional variables (more details in next section) that extends to modal formulas as follows:
\begin{align*}
e(w,\Diamond \varphi ) &= \sup_{w'\in W}\{\min(R(w,w'), e(w',\varphi ))\} \\
e(w,\Box \varphi )	&= \inf_{w'\in W}\{R(w,w')\Rightarrow e(w',\varphi )\}.
\end{align*}

We will denote by ${\cal K}45(\mathbf{G})$ the class of $[0, 1]$-models $\m =  \langle W, R, e \rangle$ where $R$ satisfies the following many-valued counterpart of the classical properties:

\begin{itemize}
\item {\em Transitivity}: $\forall w, w', w'' \in W$,  $\min(R(w, w'), R(w', w'')) \leq R(w, w'')$
\item {\em Euclidean}: $\forall w, w', w'' \in W$,  $\min(R(w, w'), R(w, w'')) \leq R(w', w'')$
\end{itemize}

In this setting,  the class $\Pi\mathcal{G}$ of {\em possibilistic Kripke 
models} $\langle W, \pi, e \rangle$, where $\pi \colon W \to [0, 1]$ is a possibility distribution (not necessarily normalized) on the set of worlds $W$, can be considered as the subclass of ${\cal K}45(\mathbf{G})$ models $\langle W, R, e \rangle$ where $R$ is independent of the world in its first argument, in the sense that $R(w, w') = \pi(w')$.  Since  $\Pi\mathcal{G} \subsetneq {\cal K}45(\mathbf{G})$, it follows that the set $Val({\cal K}45(\mathbf{G}))$ of valid formulas in the class of 
${\cal K}45(\mathbf{G})$ is a subset of the set $Val(\Pi\mathcal{G} )$ of valid formulas in the class $\Pi\mathcal{G}$, i.e. $Val({\cal K}45(\mathbf{G})) \subseteq Val(\Pi\mathcal{G} )$. The interesting question is therefore whether the converse inclusion  $Val(\Pi\mathcal{G} ) \subseteq Val({\cal K}45(\mathbf{G}))$ holds, and thus whether the class $\Pi\mathcal{G}$ provides a simplified possibilistic semantics for the modal logic $K45(\mathbf{G})$.

In \cite{BEGR16} the authors claim to provide a positive answer, not for the class ${\cal K}45(\mathbf{G}))$ models but for the subclass of ${\cal KD}45(\mathbf{G}))$ models, i.e.\ those   ${\cal K}45(\mathbf{G}))$ models $\m = \langle W, R, e \rangle$ further satisfying the many-valued counterpart of seriality:
\begin{itemize}
\item{\em Seriality}: $\forall w \in W$, $\sup_{w' \in W} R(w, w') = 1$
\end{itemize}
Indeed, they prove that the logic $KD45(\mathbf{G})$, complete w.r.t. the 
class of  ${\cal KD}45(\mathbf{G}))$ models, is also complete w.r.t. the class $\Pi^*\mathcal{G}$ of possibilistic models $\langle W, \pi, e \rangle$, where $\pi: W \to [0, 1]$ is a {\em normalized} possibility distribution on $W$. 
%
Unfortunately, it has to be noted that the completeness proof in \cite{BEGR16} has some flaws. 
In this paper we provide a correct proof, not only for the completeness of $KD45(\mathbf{G})$ w.r.t. to its corresponding class of possibilistic frames, but also for the weaker logic $K45(\mathbf{G})$ accounting for partially inconsistent possibilistic Kripke frames. 


{This paper is organized as follows: Section~\ref{sec:preliminaries} introduces the main notions about minimum propositional modal G\"odel logic  and its relational semantics;
Section~\ref{sec:K45logic} then formally introduces a calculus for the logic $K45(G)$, some of its extensions and several of its main theorems;
Section~\ref{sec:possibilitysem} presents our simplified possibilistic semantics and the results of completeness; 
Section ~\ref{sec:decidability} is devoted to analyze the finite model property for the new semantics;
finally, Section~\ref{sec:Disc-concl} provides some conclusions.}

\section{Preliminaries on propositional and modal G\"odel fuzzy logic} \label{sec:preliminaries}

This section is devoted to preliminaries on the G\"odel fuzzy logic G. We present their syntax and semantics, their main logical properties and the notation we use throughout the article.

The language of G\"odel propositional logic $\mathcal{L}(V)$ is built as usual from a countable set of propositional variables $V$, the constant $\gbot$ and the binary connectives $\land$ (conjunction) and $\to$ (implication). 
%
Further connectives are defined as follows: 
\begin{eqnarray*}
\gtop & \mbox{is} & \gbot \rightarrow \gbot \\
\varphi \vee \psi & \mbox{is} & ((\varphi \rightarrow \psi)\rightarrow
\psi)
\wedge ((\psi \rightarrow \varphi )\rightarrow \varphi),\\
\lnot \varphi & \mbox{is} & \varphi \rightarrow \gbot,\\
\varphi \equiv \psi & \mbox{is} & (\varphi \rightarrow \psi) \wedge (\psi
\rightarrow
\varphi).
\end{eqnarray*}
The following are the {\em axioms\/} of $G$: \\

\begin{tabular}{l l}
(A1) &$(\varphi \rightarrow \psi) \rightarrow ((\psi \rightarrow
\chi) \rightarrow (\varphi \rightarrow \chi))$
\\
(A2) &$(\varphi \wedge \psi) \rightarrow \varphi$
\\
(A3) &$(\varphi \wedge \psi) \rightarrow (\psi \wedge \varphi)$
\\
(A4a) &$(\varphi \rightarrow (\psi \rightarrow \chi ))
\rightarrow ((\varphi \wedge \psi) \rightarrow \chi)$
\\
(A4b) &$((\varphi \wedge \psi) \rightarrow \chi) \rightarrow
(\varphi\rightarrow (\psi \rightarrow \chi))$
\\
(A6) & $\varphi \to (\varphi \wedge \varphi)$
\\
(A7) & $((\varphi \rightarrow \psi) \rightarrow \chi) \rightarrow
 (((\psi \rightarrow \varphi ) \rightarrow
\chi) \rightarrow \chi)$
\\
(A8) &$\gbot \rightarrow \varphi$
\\
\end{tabular}
\ \\

\noindent The {\em deduction rule\/} of $G$ is modus ponens. \\

G\"odel logic can be obtained in fact as the axiomatic extension of H\'ajek's Basic Fuzzy Logic BL \cite{Hajek98} (which is the logic of continuous t-norms and their residua)  by means of the contraction axiom (A6) $\varphi \to (\varphi \land \varphi)$. Since the unique idempotent continuous t-norm is the minimum, this yields that G\"odel logic is strongly complete with respect to its standard fuzzy semantics that interprets formulas over the structure $[0, 1]_\mathrm{G} := \langle [0, 1], \min, \Rightarrow_\mathrm{G}, 0, 1 \rangle$\footnote{Called standard G\"odel algebra.}, i.e. semantics defined by truth-evaluations $e$ such that $e(\varphi \land \psi) = \min(e(\varphi), e(\psi))$,  $e(\varphi \to \psi) = e(\varphi) \Rightarrow_\mathrm{G} e(\psi)$ and $e(\gbot) = 0$.

G\"odel logic can also be seen as the axiomatic extension of intuitionistic propositional logic  by the prelinearity axiom $(\varphi \to \psi) \lor (\psi \to \varphi)$.  Its algebraic semantics is therefore given by the variety of prelinear Heyting algebras, also known as G\"odel algebras. In fact,  it is sound and complete in the following stronger sense, see \cite{CaiRod2010}.

\begin{proposition} \label{ordersoundness}
\begin{itemize}
\item[i)] If $T\cup \{\varphi \}\subseteq \mathcal{L}(X),$
then $T\vdash _G \varphi $ implies $\inf v(T)\leq v(\varphi )$
for any valuation $v:X\rightarrow \lbrack 0,1]$.

\item[ii)]If $T$ is countable,
and $T\nvdash _ G \varphi _{i_{1}}\vee ..\vee \varphi _{i_{n}}$
for each finite subset of a countable family $\{\varphi _{i}\}_{i\in I}$ there is
an evaluation $v:\mathcal{L}(X) \rightarrow \lbrack 0,1]$\ such that $v(\theta )=1$\ for all
$\theta \in T$\ and $v(\varphi _{i})<1$ for all $i \in I$.

\end{itemize}
\end{proposition}

We mention in passing that the  algebraic  semantics of  G\"odel algebra is is  given by the variety of prelinear Heyting algebras, also known as G\"odel algebras. A G\"odel algebra is a structure ${\bf A} = \langle A, *, \Rightarrow, 0, 1 \rangle$ which is a (bounded, integral, commutative) residuated lattice satisfying the contraction equation $$x * x = x,$$ and pre-linearity equation $$(x \Rightarrow y) \lor (y \Rightarrow x) = 1, $$where $x \lor y  = ((x \Rightarrow y) \Rightarrow y) * ((y \Rightarrow x) \Rightarrow x)$). \vspace{5mm}

Let us consider a modal expansion of G\"odel logic with two operators $\Box$ and $\Diamond$.  The set of formulas $\mathcal{L}_{\Box \Diamond }(V)$ 
is built as $\mathcal{L}(V)$ (always assuming countability of the set of propositional variables $V$) but extending the set of operations with two unary symbols $\Box $ and $\Diamond$. Whenever $V$ is clear from the context we will simply write $\mathcal{L}_{\Box \Diamond}$.

In the style introduced by Fitting \cite{Fitting91,Fi92b} and studied in the works mentioned in the introduction, we define the G\"odel Modal Logic as arising from its semantic definition. This is given by enriching usual Kripke models with evaluations over the previous standard algebra, as in \cite{CaiRod2010, CaiRod2015} and others. Formally:

\begin{definition}
	A \termDef{G\"odel-Kripke model}  is a structure $\langle W, R, e\rangle$ where $W$ is a non-empty set of so-called worlds, and $R \colon W \times W \rightarrow[0,1]$ and $e \colon V \times W \rightarrow [0,1]$ are arbitrary mappings.
\end{definition}

The evaluation $e$ can be uniquely extended to a map with domain $W \times \mathcal{L}_{\Box\Diamond}$ in such a way that it is a propositional G\"odel homomorphism (for the propositional connectives) and where the modal operators are interpreted as infima and suprema, namely:
\begin{itemize}
	\item $e(v, \bot) \coloneqq 0$,
	\item $e(v, \varphi \star \psi) \coloneqq e(v, \varphi) \star e(v, \psi)$  for $\star \in \{\wedge, \vee, \rightarrow\}$,
	\item $e(v,\Box \varphi) \coloneqq \bigwedge_{w \in W}( R(v,w) \rightarrow e(w, \varphi)),$
	\item $e(v,\Diamond \varphi) \coloneqq \bigvee_{w \in W}( R(v,w) \wedge e(w, \varphi)).$
\end{itemize}
A formula $\varphi$ is \emph{valid} in a G\"odel-Kripke model $\langle W, R, e\rangle$, if $e(w,\varphi) = 1$ for all $w \in W$. We will denote by $\cal G$ the class of G\"odel Kripke models and will say that a formula $\varphi$ is \emph{$\cal G$-valid}, written $\models_{\mathcal{G}} \varphi$, if $\varphi$ is valid in all  G\"odel Kripke models.

In their paper \cite{CaiRod2015} Caicedo and Rodriguez define the logic $K(\mathbf{G})$ as the smallest set of formulas containing some axiomatic version of G\"{o}del-Dummet
propositional calculus, that is, Heyting calculus plus the prelinearity law, and the following additional axioms and rule:  \\

\begin{tabular}{rlrl}
$(K_\Box)$ & $\Box(\varphi \to \psi) \to (\Box \varphi \to \Box \psi)$ &
$(K_\Diamond)$ & $\Diamond(\varphi \lor \psi) \to (\Diamond\varphi  \lor \Diamond \psi)$\\
$(F_\Box)$ & $\Box \gtop$ &
$(P)$ & $\Box(\varphi \to \psi) \to (\Diamond \varphi \to \Diamond \psi)$\\
$(FS2)$ & $(\Diamond\varphi  \to \Box \psi) \to \Box(\varphi \to \psi)$ \mbox{} \mbox{} \quad &
$(Nec)$ & from $\varphi$ infer $\Box \varphi$ \\
\end{tabular}

\vspace{5mm}
 They denote deduction in this system as  $\vdash _{K(\mathbf{G})}$ and  they show the following result:

\begin{theorem}[{\cite[Theorem 3.1]{CaiRod2015}}]\label{th:compGK}
Let $\varphi \in \mathcal{L}_{\Box\Diamond}$. Then
\[\vdash_{K(\mathbf{G})} \varphi \text{ if and only if }  \models_{\mathcal{G}} \varphi.\]
\end{theorem}

\vspace{0.2cm} 

\noindent Proofs with assumptions are allowed with the restriction that $(Nec)$ may be applied only when the premise is a
theorem. 
This restriction allows for the following convenient reduction (see \cite{CaiRod2010}).\medskip

\begin{lemma}
{\label{reduction} } Let $ThK(\mathbf{G})$ be the set of theorems of $K(\mathbf{G})$ with no assumptions, then for any theory $T$ and formula $\varphi $ in $\mathcal{L}_{\Box \Diamond
}:T\vdash _{K(\mathbf{G})}\varphi $\ if\ and\ only\ if\ $ T\cup ThK(G)\vdash _{\mathbf{G}}\varphi $ where  $\vdash_{\mathbf{G}}$ denotes deduction in G\"odel fuzzy logic G.
\end{lemma}

%
%
%

\noindent The following are some theorems of $K(\mathbf{G})$,  see \cite{CaiRod2015}:  \vspace{2mm}

$%
\begin{array}{rlrl}
T1. & \lnot \Diamond \theta \leftrightarrow \square \lnot \theta &  \\ 
T2. & \lnot \lnot \square \theta \rightarrow \square \lnot \lnot \theta & 
\\ 
T3. & \Diamond \lnot \lnot \varphi \rightarrow \lnot \lnot \Diamond \varphi &  \\
T4. & (\square \varphi \rightarrow \Diamond \psi )\vee \square ((\varphi \rightarrow \psi )\rightarrow \psi )\\
T5. & \Diamond (\varphi \to \psi) \to (\square \varphi \to \Diamond \psi)
\\
\end{array}
$ \vspace{2mm}

\noindent The first one is an axiom in Fitting's systems in \cite{Fitting91}, the next two were
introduced in \cite{CaiRod2015}, the fourth one will be useful in our completeness
proof and is the only one depending on prelinearity. The last is known as 
the first connecting axiom given by Fischer Servi.  
{In addition, it is interesting to notice that the following rule is derivable:\vspace{2mm}  

$
\begin{array}{rl}
(Nec_\Diamond) & \mbox{from } \varphi \to \psi  \mbox{ infer } \Diamond \varphi \to \Diamond \psi \\
\end{array}
$ \vspace{2mm}

\noindent Indeed,  if $\vdash _{K(\mathbf{G})} \varphi \to \psi$, then  $\vdash _{K(\mathbf{G})} \Box(\varphi \to \psi)$ by $Nec$, and by using axiom $(P)$ and modus ponens we  get $\vdash _{K(\mathbf{G})} \Diamond \varphi \to \Diamond \psi$. 
}

In the next section, we will focus on an extension of $K(\mathbf{G})$ for which we are able to simplify the G\"odel Kripke semantics introduced in this section.

\section{The logic $K45(\mathbf{G})$} \label{sec:K45logic}
 
\noindent  The logic $K45(\mathbf{G})$ is defined by adding to  $K(\mathbf{G})$ the following axioms: \vspace{0.1cm}

\begin{tabular}{ll@{\qquad}ll}
\hspace{2mm} $(4_\Box)$& $\Box \varphi \to \Box \Box \varphi$ & \hspace{2.3cm} $(4_\Diamond)$& $\Diamond\Diamond \varphi \to  \Diamond \varphi$ \\
\hspace{2mm} $(5_\Box)$&  $\Diamond\Box \varphi \to  \Box \varphi$ &\hspace{2.3cm} $(5_\Diamond)$& $\Diamond \varphi \to \Box \Diamond \varphi$ \\
\end{tabular}
\vspace{0.2cm}

\noindent We define the logic $KD45(\mathbf{G})$ by adding to $K45(\mathbf{G})$ the following axiom: 

\begin{tabular}{rl@{\qquad}ll}
\hspace{2mm} $(D)$ &  \blue{$\Diamond \gtop$} & \ \\
\end{tabular} 
\vspace{0.2cm}

\noindent and $KT45(\mathbf{G})$ if we add the following axioms to  $K45(\mathbf{G})$: \vspace{0.1cm}

\begin{tabular}{ll@{\qquad}lll}
\hspace{2mm} $(T_\Box)$& $\Box \varphi \to \varphi$ & \hspace{2,8cm} $(T_\Diamond)$& $\varphi \to \Diamond \varphi$ \\
\end{tabular}
\vspace{0.2cm}

On the other hand, the following is a theorem of $KD45(\mathbf{G})$:  \vspace{2mm}

$
\begin{array}{rl}
(D') & \Box \varphi \to \Diamond \varphi  \\ 
\end{array}
$  \vspace{2mm}

\noindent Indeed, we can replace $\gtop$  by $\varphi \to \varphi$ in Axiom (D) and then use T5. In fact, (D) and (D') are interderivable in $K(\mathbf{G})$, that is,  we have both $ \Diamond \gtop  \vdash_{K(\mathbf{G})} 
\Box \varphi \to \Diamond \varphi $ and $\Box \varphi \to \Diamond \varphi  \vdash_{K(\mathbf{G})} \Diamond \gtop$, the latter holding by first instantiating (D') with $\varphi = \gtop$ and getting $\Box \gtop \to \Diamond \gtop$,  and then using that $\Box \gtop$ is an axiom of $K(\mathbf{G})$.

Next we show that in $K45(\mathbf{G})$  some iterated modalities can be simplified. This is in accordance with our intended simplified semantics for $K45(\mathbf{G})$ that will be formally introduced in the next section.

\begin{proposition} \label{simplif}
The logic $K45(\mathbf{G})$ proves the following schemes:
\vspace{0.2cm}

$
\begin{array}{llll}
\hspace{2mm} ({F}_{\Diamond\Box}) & \Diamond\Box\gtop \leftrightarrow \Diamond \gtop & 
\hspace{2cm} {(G45)} & 
{(\bo \varphi \to \di \psi) \to \bo (\bo \varphi \to \di \psi)} \\
\hspace{2mm} ({U}_\Diamond)  &  \Diamond\Diamond  \varphi \leftrightarrow 
 \Diamond  \varphi & \hspace{2cm} ({U}_\Box) & \square \square\varphi \leftrightarrow \square\varphi  \\
\hspace{2mm} (T4_\Box) &  (\Box \varphi \to \Diamond \Box\varphi) \vee \Box \varphi & \hspace{2cm} (T4_\Diamond)   & (\Box\Diamond \varphi \to \Diamond \varphi) \vee \Box\Diamond \varphi \\
\hspace{2mm} (Sk_\Diamond) & (\Diamond \gtop \to \Diamond \varphi) \leftrightarrow \Box\Diamond \varphi &  \hspace{2cm} (T4'_\Diamond)   & (\Box\Diamond \varphi \to \Diamond \varphi) \vee  (\Diamond \gtop \to \Diamond \varphi)

\end{array}$\\

\end{proposition}

\begin{proof2}
$({F}_{\Diamond\Box})$ is an immediate consequence of $F_\Box$ and $Nec_\Diamond$. 
As for schemes $U_\Diamond$ and $U_\Box$,  axioms $4_\Box$ and $4_\Diamond$ give one direction of them. The opposite directions, together with the rest of schemes, are obtained as follows:\\

\vspace{-2em}
\[
\begin{array}{ll}
\text{Proof $(U_\Diamond)$:} & \text{Proof $(U_\Box)$:} \\[0.4em]
	\begin{array}{@{}ll@{}}
	\Diamond \varphi \to \Box \Diamond \varphi  & \text{ axiom } (5_\Diamond) \\
	\Box(\varphi \to \Diamond\varphi) & \text{ by } MP \text{ and }  (FS2) \\
	\Diamond \varphi \to \Diamond \Diamond \varphi & \text{ by } MP \text{ and }  (P)
	\end{array}	
&
	\begin{array}{@{}ll@{}}
	\Diamond \Box\varphi \to \Box \varphi & \text{ axiom } (5_\Box) \\
	\Box(\Box \varphi \to \varphi) & \text{ by } MP \text{ and }  (FS2) \\
	\Box \Box \varphi \to \Box \varphi & \text{ by } MP \text{ and }  (K_\Box)
	\end{array}
\\[2em]
\mbox{Proof $(T4_\Box)$:} 	& \mbox{Proof $(T4_\Diamond)$:} \\
	\begin{array}{@{}ll@{}}
	\Box\Box \varphi \to \Diamond \Box \varphi \vee \Box\Box\varphi		& \text{(T4) with } \Box\varphi \\
	\Box \varphi \to \Diamond \Box \varphi \vee \Box\varphi		& \text{by } (U_\Box)
	\end{array}
&
	\begin{array}{@{}ll@{}}
	(\Box\Diamond \varphi \to \Diamond \Diamond \varphi) \vee \Box((\Diamond \varphi \to \Diamond \varphi) \to \Diamond \varphi)		& \text{(T4)} \\
	(\Box\Diamond \varphi \to \Diamond \Diamond \varphi) \vee  \Box(\gtop \to \Diamond \varphi)	& \text{equiv}
	\end{array}
\\[2em]
\text{Proof $(G45)$:} & \\[0.4em]
\multicolumn{2}{l}{
	\begin{array}{@{}ll@{}}
	(\bo \varphi \to \di \psi) \to (\bo \varphi \to \bo \di \psi) 	& \text{by applying } (5_\di) \\
	(\bo \varphi \to \bo \di \psi) \to (\di \bo \varphi \to \bo \di \psi) & \text{by applying } (5_\bo) \\
	(\di \bo \varphi \to \bo \di \psi) \to \bo(\bo \varphi \to \di \psi) & \text{by } (FS2)
	\end{array}
}
\\[2em]
\text{Proof $(Sk_\Diamond)$:} & \\[0.4em]
\multicolumn{2}{l}{
	\begin{array}{@{}ll@{}}
	\Box(\gtop \to \Diamond \varphi) \to (\Diamond \gtop \to \Diamond\Diamond \varphi) 	& \text{by (P)} \\
	\Box \Diamond \varphi \to (\Diamond \gtop \to \Diamond \varphi)		& \text{by $(U_\Diamond)$ and equivalences} \\
	(\Diamond \gtop \to \Box \Diamond \varphi) \to \Box (\gtop \to \Diamond \varphi) & \text{by } (FS2) \\
	(\Diamond \gtop \to \Diamond \varphi) \to \Box \Diamond \varphi		& \text{by } (5_\Diamond)
	\end{array}
}
\\[2.4em]
\text{Proof $(T4'_\Diamond)$: using } (T4_\di) \text{ and } (Sk_\Diamond)
\end{array}
\]
\vspace{-3mm}

\end{proof2}


%
%

{Moreover, if we restrict ourselves to formulas starting with $\Box$ or $\Diamond$ we can prove the following property. 
\begin{lemma} \label{rmk} Let $X = \{\Box \theta, \Diamond \theta : \theta \in \mathcal{L}_{\Box,\Diamond}\}$. If $\varphi \in X$ then the schemas
$$ \varphi \to  \Box\varphi \hspace{5mm} \mbox{ and } \hspace{5mm} \Diamond \varphi \to  \varphi$$ 
are theorems of $K45(\mathbf{G})$.
\end{lemma}	

\begin{proof} We check that $K45(\mathbf{G})$ derives $\varphi \to  \Box\varphi$ if  $\varphi \in X$, the other schema is similar. In fact, we have two cases: if $\varphi = \Box \psi$, then $\varphi \to  \Box\varphi$ is in fact one direction of  $(U_\Box)$; if $\varphi = \Diamond \psi$, then  this is axiom $(5_\Diamond)$. 
\end{proof}
}

%
%

From now on we will use $ThK45(\mathbf{G})$ to denote the set of theorems of $K45(\mathbf{G})$. We close this section with the following observation: 


\begin{lemma} \label{reduction2} If $T$ is a finite set of formulas,  $T\vdash _{K45(\mathbf{G})} \varphi $ iff $\vdash _{K45(\mathbf{G})} T^{\land} \to \varphi $, where $T^{\land} = \bigwedge\{ \psi \mid \psi \in T \}$.
\end{lemma}

\begin{proof2}  By Lemma \ref{reduction}, we have $T\vdash _{K45(\mathbf{G})} \varphi $ iff $T\cup ThK45(\mathbf{G}) \vdash _{G} \varphi $. By the deduction theorem of G\"odel logic, the latter is equivalent to $ ThK45(\mathbf{G}) \vdash _{G}  T^{\land} \to \varphi $, and by (i) again, this is equivalent to $\vdash _{K45(\mathbf{G})} T^{\land} \to \varphi $. 
\end{proof2}

\begin{remark} \label{emptyzone} It is worth noting that for any valuation $v$ such that $v(ThK45(\mathbf{G})) = 1$ there is no formula $\varphi$ such that $v(\Diamond \gtop) < v(\nabla \varphi) < 1$ with $\nabla \in 
\{\Box , \Diamond\}$ because both formulae $(\Box \varphi \to \Diamond \varphi) \vee \Box \varphi$ and $\Diamond \varphi \to \Diamond \gtop$ are in 
$ThK45(\mathbf{G})$.
\end{remark}

\section{Simplified Kripke semantics and completeness} \label{sec:possibilitysem}

In this section we will show that $K45(\mathbf{G})$  is complete with respect to a class of simplified Kripke G\"odel frames.

\begin{definition}
\label{simplgodelframe} A {\em possibilistic Kripke frame}, or $\Pi$-frame, is a
structure $\langle W,\pi \rangle $ where $W$ \emph{is a non-empty
set of worlds}, and $\pi:W \rightarrow [0,1]$ is a {\em possibility distribution} over $W$.

A {\em possibilistic G\"odel Kripke model}, \emph{$\Pi G$-model} for short, is a triple $ \langle W, \pi, e\rangle$ where $ \langle W, \pi\rangle $ is a $\Pi$-frame and $e: W \times Var \to [0, 1]$ provides a G\"odel  evaluation of variables in each world.
For each $w \in W$, $e(w, -)$ extends to arbitrary formulas in the usual way for the propositional connectives and for modal operators in the following way:

\medskip

\hspace{3.5cm} $e(w,\Box \varphi ):=\inf_{w'\in W}\{\pi(w') \Rightarrow 
e(w',\varphi )\}$

\hspace{3.5cm} $e(w,\Diamond \varphi ):=\sup_{w'\in W}\{\min(\pi(w'), e(w',\varphi ))\}$.

\medskip
\noindent If $\pi$ is normalised, i.e. if $\sup_{w \in W} \pi(w) = 1$, then $ \langle W, \pi, e\rangle$ will be called a {\em normalised} possibilistic G\"odel Kripke model, or \emph{$\Pi^* G$-model}. A formula $\varphi$ is \emph{valid} in a $\Pi G$-model $\langle W, \pi, e \rangle$ if $e(w,\varphi) = 1$ for all $w \in W$.

We will denote by $\Pi \cal G$ the class of possibilistic G\"odel Kripke models, and by $\Pi^* \cal G$ the subclass of normalised models. We say that a formula $\varphi$ is \emph{$\Pi \cal G$-valid}, written $\models_{\Pi\mathcal{G}} \varphi$, if $\varphi$ is valid in all possibilistic G\"odel Kripke models, and \emph{$\Pi^* \cal G$-valid}, written $\models_{\Pi^*\mathcal{G}} \varphi$, if $\varphi$ is valid in all normalised possibilistic G\"odel Kripke model.
\end{definition}

Observe that the evaluation of formulas beginning with a modal operator is in fact independent from the current world. Also note that the $e(\cdot,\Box \varphi )$ and $e(\cdot,\Diamond \varphi )$ are in fact generalisations for G\"odel logic propositions of the necessity and possibility degrees of $\varphi$ introduced in Section \ref{intro}  for classical propositions, although now they are not dual (with respect to G\"odel negation) any longer.

%



In the rest of this section we are going to show  in detail a weak completeness proof of the logic $K45(\mathbf{G})$ (resp.\ $KD45(\mathbf{G})$) with respect to the class $\Pi\mathcal{G}$ (resp.\ the subclass $\Pi^*\mathcal{G}$) of possibilistic G\"odel Kripke models.  In fact one can prove a slightly stronger result, namely completeness for deductions from finite theories.

We start with the case of $K45(\mathbf{G})$. In what follows, for any formula $\varphi$  we denote by $Sub(\varphi) \subseteq \mathcal{L}_{\square 
\Diamond }$
the set of subformulas of $\varphi$ and containing
the formula $\gbot $.  Moreover, let $X:=\{\square \theta ,\Diamond \theta :\theta \in \mathcal{L}_{\square
\Diamond }\}$ be the set of formulas in $\mathcal{L}_{\square \Diamond }$
beginning with a modal operator; then $\mathcal{L}_{\square \Diamond }(Var)=%
\mathcal{L(}Var\cup X)$. That is, any formula in $\mathcal{L}_{\square
\Diamond }(Var)$ may be seen as a propositional G\"odel  formula built from the extended set
of  propositional variables $Var\cup X$. In addition, for a given formula 
$\varphi$, 
let $\sim_\varphi$ be the equivalence relation on $[0,1]^{Var\cup X} \times \lbrack 0,1]^{Var\cup X}$
defined as follows:
$$ u \sim_\varphi w \mbox{ iff } \forall \psi \in Sub(\varphi): u(\Box \psi) = w(\Box \psi) \mbox{ and } u(\Diamond \psi) = w(\Diamond \psi) .$$

Now, assume that a formula $\varphi$ is not a theorem of $K45(\mathbf{G})$. Hence by completeness of G\"odel calculus and Lemma \ref{reduction}, there exists a G\"odel valuation $v$  such that $v(ThK45(\mathbf{G}))=1$ and $v(\varphi)<1$.
With the valuation $v$ now fixed, we follow the usual canonical model construction, defining a canonical $\Pi G$-model $\m^v_{\varphi}$ in which we will show $\varphi$ is not valid.

The \emph{canonical model} \emph{\ }$%
\m^v_{\varphi}= \langle W^v_{\varphi},\pi^v_{\varphi},e^v_{\varphi} \rangle$ is defined 
as follows:
\begin{itemize}
\item $W^v_{\varphi}$ is the set $\{u \in \lbrack 0,1]^{Var\cup X} \mid u 
\sim_\varphi v \mbox{ and } u(ThK45(\mathbf{G}))=1 \}$.
\item $\pi^v_{\varphi}(u)=\inf_{\psi \in Sub(\varphi)}\{\min(v(\Box \psi)\rightarrow u(\psi ), u(\psi )\rightarrow v(\Diamond \psi ))\}.$
\item $e^v_{\varphi}(u,p)=u(p)$ for any $p\in Var$.
\end{itemize}

In this context, we call the elements of $\Delta_\varphi := \{\square \theta ,\Diamond \theta :\theta \in Sub(\varphi) \}$ the {\em fixed points} of the Canonical Model.
Note that having $\nu(ThK45(\mathbf{G}))=1$ does not guarantee that $\nu$ belongs to the canonical model because it may not take the appropriated values for the fixed points, i.e.\ it may be that  $u \not\sim_\varphi \nu$. However, the next lemma shows how to, under certain conditions, transform such an evaluation into another  belonging to the canonical model.

\begin{lemma}\label{normalization} Let $u \in W^v_\varphi$ and let $\nu: {Var\cup X} \to [0,1]$ be a G\"odel valuation. Define $\delta = \max \{ u(\lambda) : \nu(\lambda) < 1 \mbox{ and } \lambda \in \Delta_\varphi\}$ and  $\beta = \min \{ u(\lambda) : \nu(\lambda)= 1 \mbox{ and } \lambda \in \Delta_\varphi\}$. If $\nu$ satisfies the following conditions:

\begin{description}
\item [a.] $\nu(ThK45(\mathbf{G}))=1$.
\item [b.] For all $\lambda \in X$, we have $u(\lambda) > \delta \Rightarrow \nu(\lambda)=1$.
\item [c.] For any $\psi, \phi \in \{\lambda : u(\lambda) \leq \delta \mbox{ and } \lambda \in X \}$: $u(\psi) < u(\phi)$ implies $\nu(\psi) < \nu(\phi)$.
\item [d.] {For any $\psi, \phi \in \Delta_\varphi: u(\psi) \leq u(\phi)$ implies $\nu(\psi) \leq \nu(\phi)$.}
\end{description}
then, there exists a G\"odel valuation $w \in W^v_\varphi$ such that for any $\varepsilon > 0$ with $\delta+\varepsilon < \beta$, and for any formulae $\psi$ and $\phi$, the following conditions hold:
\begin{enumerate}
\item $\nu(\psi) = 1 $  implies $ w(\psi) \geq \delta+\varepsilon$.
\item $\nu(\psi) < 1 $ implies $ w(\psi) < \delta+\varepsilon$.
\item $1 \neq  \nu(\psi) \leq \nu(\phi)$ implies $w(\psi) \leq w(\phi)$.
\item $\nu(\psi) < \nu(\phi)$ implies $ w(\psi) < w(\phi)$.
\end{enumerate}
\end{lemma}

\begin{proof2} 

First of all, 
notice that if $\nu$ satisfies the conditions {\bf c} and {\bf d}, then necessarily $\delta < \beta$. Indeed, suppose $\delta \geq \beta$. Then there are at least two formulas $\theta_1$ and $\theta_2$ in $\Delta_\varphi$ such that $\nu(\theta_1) <1$, $\nu(\theta_2)=1$ and $\delta \geq u(\theta_1) \geq u(\theta_2) \geq \beta$. Note that the case $u(\theta_1) > u(\theta_2)$ is not possible because it would violate condition {\bf c}, and the case  $u(\theta_1) = u(\theta_2)$ is also imposible because it would then violate condition {\bf d}.


Let  $B= \{ \nu(\lambda) : \lambda \in \Delta_\varphi , \nu(\lambda) < 1\} \cup \{0 \} = \{ b_0 = 0 < b_1 < \ldots < b_N \}$. Obviously, $b_N < 1$. For each $0 \leq i \leq N$, pick  $\lambda_i \in \Delta_\varphi$ such that $\nu(\lambda_i) = b_i$.
Define now a continuous strictly function $g:[0,1]\mapsto \lbrack 0,\delta+\varepsilon) \cup \{1\}$ such
that \medskip

$g(1)=1$

$g(b_i)= u(\lambda_i)$ for every $0 \leq i \leq N$

$g[(b_N, 1)]= (\delta, \delta+\varepsilon)$
\medskip

\noindent Notice that $\delta = g(b_N)$.  In addition, define another continuous strictly increasing function $h:[0,1] \mapsto [\delta+\varepsilon, 1]$ such that
\medskip

$h(0)=\delta+\varepsilon$

$h[(0,\beta)] =(\delta+\varepsilon, \beta)$

$h(x) = x$, for $x \in [\beta, 1]$
\medskip

\noindent Then we define the valuation $w:  Var \cup X \to [0,1]$ as follows:

\medskip
$w(p) = \left \{
\begin{array}{ll}
g(\nu(p)), & \mbox{if } \nu(p) < 1,\\
h(u(p)), & \mbox{if } \nu(p) = 1.
\end{array}
\right .
$
\medskip

Next step is to prove that $w$ satisfies the required Properties 1--4. And we are going to prove it by induction on {the maximum of the complexity of both formulae $\psi, \phi$.}

First, we consider the base case when both $\psi$ and $\phi$ belong to $Var \cup X$. Then
\begin{enumerate}
    \item By definition of $w$, $\nu(\psi) = 1$ implies $w(\psi)= h(u(\psi)) \geq \delta+\varepsilon$. The same happens for $\phi$.
    \item By definition of $w$, $\nu(\psi) < 1$ implies $w(\psi)= g(\nu(\psi)) < \delta+\varepsilon$. The same happens for $\phi$.
    \item Suppose $1 \neq \nu(\psi) \leq \nu(\phi)$. Since $\nu(\psi) < 1$, by definition of $w$, we have $w(\psi)=g(\nu(\psi)) < \delta+\varepsilon$. Now, we analyse different cases.
  If $\nu(\phi) < 1$ then $w(\phi) = g(\nu(\phi)) \geq g(\nu(\psi)) = 
w(\psi)$ because $g$ is strictly increasing. Otherwise, If $\nu(\phi) =1$ then $w(\phi)= h(u(\phi)) \geq \delta +\varepsilon > w(\psi)=g(\nu(\psi))$.
  \item  Suppose $\nu(\psi) < \nu(\phi)$. In this case, the proof is similar to the previous one. If $\nu(\psi) < 1$ and $\nu(\phi)=1$ then $w(\phi)= h(u(\phi)) \geq \delta +\varepsilon > g(\nu(\psi)) = w(\psi)$. On 
the contrary, if both $\nu(\psi)<1$ and $\nu(\phi)<1$ then $w(\phi) = g(\nu(\phi)) > g(\nu(\psi)) = w(\psi)$ because $g$ is strictly increasing.
This case will become important when we try to prove that $w$ satisfies the axioms of $K45(\mathbf{G})$. \label{extraproperty}
 \end{enumerate}
 In fact, in the case both formulae $\psi$ and $\phi$ belong to $Var \cup 
X$  we can prove that two further conditions hold:
 \begin{enumerate}
 \item[5.] If $\nu(\psi)=\nu(\phi)=1$ and $u(\psi) \leq u(\phi)$ then 
$w(\psi) \leq w(\phi)$.  \label{extraproperty1}
 \item[6.] If $\nu(\psi)=\nu(\phi)=1$ and $u(\psi) < u(\phi)$ then $w(\psi) < w(\phi)$.  \label{extraproperty2}
 \end{enumerate}
 Indeed, by definition of $w$, we have $w(\psi)= h(u(\psi))$ and $w(\phi)= h(u(\phi))$, and since $h$ is strictly increasing, $u(\psi) \leq u(\phi)$ (resp. $u(\psi) < u(\phi)$) implies $h(u(\psi)) \leq h(u(\phi))$ (resp. $h(u(\psi)) < h(u(\phi))$), as desired. These properties will become important when proving that $w$ satisfies the axioms of $K45(\mathbf{G})$.

Now, we consider the inductive step. \medskip

\noindent {\bf Claim}: {\em Induction hypothesis (IH):  if Properties 1--4 are satisfied by formulas with complexity at most $n$, then these properties also hold for formulae with complexity at most $n+1$.}  \medskip

\noindent The proof of this claim is quite technical and it is moved to the appendix. 

Finally, we prove that $w \in W^v_{\varphi}$, i.e. we prove that both $w \sim_\varphi v$ and  $w(ThK45({\cal G})) = 1$.
\begin{itemize}
\item[(i)] By definition, $w \sim_\varphi v$ iff $w(\Box\phi) = v(\Box\phi)$ and $w(\Diamond\phi) = v(\Diamond\phi)$ for all $\phi \in Sub(\varphi)$. Let $A \in \{ \Box \phi, \Diamond \phi\} \subseteq \Delta_\varphi$.
By definition of $w$, $w(A) = g(\nu(A))$ if $\nu(A) < 1$, and $w(A) = 
h(u(A))$, otherwise. Since $A \in \Delta_\varphi$, if $\nu(A) < 1$ then, by construction ($u(A) \leq \delta$), we have $g(\nu(A)) = u(A) \leq \delta$, and if $\nu(A) =1$, again by construction ($(u(A) \geq \beta$), we have  $ h(u(A)) = u(A)$. So, we have proved that  $w \sim_\varphi u$, but by hypothesis $u \sim_\varphi v$ as well, thus $w \sim_\varphi v$ as well.

\item[(ii)] We first prove that all axioms of $K45(\mathbf{G})$ are evaluated to 1 by $w$. The axioms of $\mathcal{G}$ are evaluated to $1$ by any G\"odel valuation. As for the specific axioms of $K45(\mathbf{G})$ (i.e.\ axioms $
(4_\Box), (4_\Diamond),  (5_\Box), (5_\Diamond)$), observe that all these 
axioms are of the form $\phi \to \psi$ for some $\phi, \psi \in X$. Then it is enough to prove that for any $ \varphi, \psi \in X$, if $ \nu(\phi \to \psi)=1$ then  $w(\phi \to \psi)=1$. 
%
Indeed, if $ \nu(\phi \to \psi) = 1$, then $ \nu(\phi) \leq \nu(\psi)$ and we have two possibilities:

\begin{itemize}

\item if $1 \neq \nu(\phi) \leq \nu(\psi)$, then by Property 3, we have $w(\phi) \leq w(\psi)$.

\item if $\nu(\phi) = \nu(\psi) = 1$, then by definition of $w$, $w(\phi) = h(u(\phi))$  and $w(\psi) = h(u(\psi))$. But since $u \in  W^v_\varphi$,
we have $u(\phi \to \psi) = 1$, because we assume that $\phi \to \psi \in K45(\mathbf{G})$ with $\phi, \psi \in X$. Thus $u(\phi) \leq u(\psi)$ and, by Property 5, we know $w(\phi) = h(u(\phi)) \leq  h(u(\psi)) = w(\psi)$.
\end{itemize}

Finally, let us consider the axioms of $K(\mathbf{G})$ (Section 2). We have to prove that $w(\chi)=1$ for each of such axioms $\chi$, knowing by assumption that $\nu(\chi) = 1$. The case of $(F_\Box)$ is easy since $(F_\Box) \in X$, and then,
by definition, $w(F_\Box)=h(u(F_\Box))=1$.  
As for axiom $(K_\Box)$, note that, using propositional reasoning, it can be equivalently expressed first as $(\Box(\phi \to \psi) \land \Box \phi) \to \Box \psi$, and then as  $(\Box(\phi \to \psi) \to \Box \psi) \lor  (\Box \phi \to \Box \psi)$. Therefore, if  $\nu(K_\Box) = 1$, it means either $\nu(\Box(\phi \to \psi) \to \Box \psi) = 1$ or $\nu(\Box \phi \to \Box \psi) 
= 1$. But these two cases concern implications of formulas from $X$, and hence we are in the same situation above as with the specific axioms of 
 $K45(\mathbf{G})$, and hence using the same reasoning, we can conclude that either  $w(\Box(\phi \to \psi) \to \Box \psi) = 1$ or $w(\Box \phi \to \Box \psi) = 1$.

The case of axioms $(K_\Diamond)$ and $(P)$ can be dealt in an analogous way, as they can be written as a disjunction of implications of formulas from $X$. So the case left is that of axiom $(FS2)$, $(\Diamond \phi \to \Box\psi) \to \Box(\phi \to \psi)$.

By hypothesis, we know $\nu(FS2) = 1$, that is $\nu(\Diamond \phi) \Rightarrow \nu(\Box \psi) \leq \nu(\Box(\phi \to \psi))$. According to that, 
we have to prove $w(FS2) = 1$ is as well.
Notice that, since  $\psi \to (\phi \to \psi)$ is a tautology, $u(\Box(\psi \to (\phi \to \psi)))=\nu(\Box(\psi \to (\phi \to \psi)))=1$. Then, by definition, $w(\Box(\psi \to (\phi \to \psi)))=1$ as well and because axiom (K) is valid for $w$, we have $w(\Box \psi \to \Box (\phi \to \psi)) = 1$, i.e. $w(\Box \psi) \leq w(\Box (\phi \to \psi))$.\\
 Now, we consider the following cases:

\begin{itemize}
\item Case $u(\Diamond \phi) \leq u(\Box \psi)$. Then $u(\Box(\phi \to \psi))=1$ which implies $\nu(\Box(\phi \to \psi))=1$. Hence, by construction, $w(\Box(\phi \to \psi))= h(u(\Box(\phi \to \psi)))=1$.

\item Case $u(\Diamond \phi) > u(\Box \psi)$. Here we distinguish three subcases:

\begin{itemize}
\item  $u(\Box \psi) \leq \delta < u(\Diamond \phi)$: by Conditions {\bf b} and {\bf c}, $\nu(\Diamond \phi) = 1$ and $\nu(\Box \psi) < 1$, respectively. Therefore, by Property 4, $w(\Box \psi) < w(\Diamond \phi)$ and 
thus $w(\Diamond \varphi \to \Box \psi) = w(\Box \psi) \leq w(\Box(\phi 
\to \psi))$ and hence $w(FS2) = 1$.

\item  $\delta < u(\Box \psi) < u(\Diamond \phi)$: by Condition {\bf b}, $1 = \nu(\Diamond \phi) = \nu(\Box \psi)\leq \nu(\Box(\phi \to \psi))$. Thus, by construction, $w(FS2) = (h(u(\Diamond \varphi)) \Rightarrow 
h(u(\Box \psi))) \Rightarrow h(u(\Box(\phi \to \psi)) = 1$.

\item  $u(\Box \psi) < u(\Diamond \phi) \leq \delta$: by Condition {\bf c}, $\nu(\Box \psi) < \nu(\Diamond \phi)$, and by Property 4, $w(\Box \psi) < w(\Diamond \phi)$, and hence  $w(\Diamond \varphi \to \Box \psi) = w(\Box \psi) \leq w(\Box(\phi \to \psi))$, and hence $w(FS2) = 1$.
\end{itemize}
\end{itemize}

\end{itemize}
{So far, we have proved that $w(\phi) = 1$ if $\phi$ is a $K45(\mathbf{G})$ axiom. To conclude the proof, we need to extend this result to the rest of the formulas in $ThK45({\mathbf{G}})$. Recall that a formula $\phi$ in the set $ThK45({\mathbf{G}})$ is either an axiom of $K45(\mathbf{G})$ or $\phi$ can be proved from the axioms and rules of inference $K45(\mathbf{G})$. In the latter case, there is non-empty finite sequence of formulae $\phi_1, \phi_2,\dots, \phi_n$, with $\phi_n = \psi$, and each $\phi_i$ is either an axiom or it has been obtained by application of an inference rule on some of the preceding formulae $\phi_1, \dots, \phi_{i-1}$. We proceed by induction on the length $n$ of the sequence. The unique interesting case is when  $\phi_n$  is obtained by applying the (Nec) rule to some $\phi_j$ with $1 \leq j < n$, i.e. $\phi_n = \Box \phi_j$. It is clear that $\phi_j \in ThK45({\mathbf{G}})$, and by the inductive hypothesis, $w(\phi_j) =1$. Of course,  $\Box \phi_j \in ThK45({\mathbf{G}})$ as well, and thus $\nu(\Box\xi) =1$ too. But since $\Box \phi_j \in X$, by definition of $w$, we hace $w(\psi) = w(\Box \phi_j) = h(u(\Box \phi_j)) = 1$. }
This finishes the proof of the lemma.
\end{proof2}

\noindent Completeness will follow from the next truth-lemma. 

\begin{lemma} [Truth-lemma] \label{equation-joint} $e^v_{\varphi}(u,\psi )=u(\psi )$ for any $\psi \in
Sub(\varphi)$ and any $u\in W^{v}_\varphi$.
\end{lemma}

\begin{proof2}
For simplicity, we will write $W^c$, $\pi^c$ and $e^c$ for $W^v_{\varphi}, \pi^v_{\varphi}$ and $e^v_{\varphi}$, respectively.  We prove the identity by
induction on the complexity of the formulas in $Sub(\varphi)$, considered 
now as elements
of $\mathcal{L}_{\square \Diamond }(Var)$. For $\gbot $ and the propositional
variables in $Sub(\varphi)$ the equation holds by definition. The only non trivial
inductive steps are:\ $e^{c}(u,\Box \psi)=u(\Box \psi)$ and $
e^{c}(u,\Diamond \psi)=u(\Diamond \psi)$ for $\Box \psi,\Diamond
\psi \in Sub(\varphi).$ By the inductive hypothesis we may assume that $
e^{c}(u^{\prime },\psi)=u^{\prime }(\psi)$ for every $u^{\prime }\in
W^c;$ thus we must prove
\begin{eqnarray}
\inf_{u^{\prime }\in W^c}\{\pi^c(u^{\prime })\Rightarrow u^{\prime }(\psi
)\}=u(\Box \psi )  \label{box} \\
\sup_{u^{\prime }\in W^c}\{\min(\pi^c(u^{\prime }), u^{\prime }(\psi
)) \}=u(\Diamond \psi )  \label{Diam}
\end{eqnarray}
By definition, $\pi^c(u^{\prime })\leq (v(\Box \psi )\Rightarrow u^{\prime
}(\psi ))$ and $\pi^c(u^{\prime })\leq (u^{\prime }(\psi )\Rightarrow
v(\Diamond \psi ))$ for any $\psi \in Sub(\varphi)$ and $u^{\prime }\in W;$
therefore, $u(\Box \psi) = v(\Box \psi) \leq (\pi^c(u^{\prime })\Rightarrow u^{\prime
}(\psi))$ and $\min(\pi^c(u^{\prime }), u^{\prime }(\psi)) \leq
v(\Diamond \psi) = u(\Diamond \psi).$ Taking  infimum over $u^{\prime }$ in the first
inequality and the supremum in the second we get
\begin{equation*}
u(\Box \psi )\leq \inf_{u^{\prime }\in W^c}\{\pi^c(u^{\prime })\Rightarrow
u^{\prime }(\psi )\}, \ \sup_{u^{\prime }\in W^c}\{\min(\pi^c(u^{\prime
}),  u^{\prime }(\psi ))\}\leq u(\Diamond \psi ).
\end{equation*}
Hence, if $u(\Box \psi )=1$ and $u(\Diamond \psi )=0$ we directly obtain (\ref
{box}) and (\ref{Diam}), respectively. Therefore, it only remains\ to prove
the next two claims for $\Box \psi ,\Diamond \psi \in Sub(\varphi)$.\medskip

\noindent \textbf{Claim 1}. \emph{If $u(\Box \psi )=\alpha <1$ then, for every $\varepsilon > 0$, there exists a valuation $w\in W^c$ such that $\pi^c(w) > w(\psi)$ and $w(\psi ) < \alpha + \varepsilon$, and thus $(\pi^c(w) \Rightarrow w(\psi ))< \alpha + \varepsilon$}.\medskip

\noindent \textbf{Claim 2. }\emph{If $u(\Diamond \psi )=\alpha >0$ then, for any small enough $\varepsilon >0,$ there exists a valuation $w'\in W^c$ such that $\min(w'(\psi ),  \pi^c(w')) \geq \alpha -\varepsilon$}.

\medskip \noindent 
The proofs of these two claims are rather involved and they can be found  in  the appendix.
\end{proof2}

\begin{theorem} [Weak completeness $K45(\mathbf{G})$]
\label{WeakCompleteness} For any formula $\varphi $ in $\mathcal{L}_{\square \Diamond }$%
, $\models _{\Pi\mathcal{G}}\varphi $ iff $\vdash _{K45(\mathbf{G})}\varphi .$
\end{theorem}

\begin{proof2} One direction is soundness, and it is easy to check that the axioms are valid in the class $\Pi\mathcal{G}$ of models.  As for the other direction, assume  $\not\vdash _{K45(\mathbf{G})}\varphi .$ Then $ThK45(\mathbf{G})\not\vdash _G\varphi $ by Lemma \ref{reduction}, and thus 
there is, by Proposition~\cite[Proposition 3.1]{CaiRod2010}, a G\"odel valuation $v:Var\cup X\rightarrow [0,1]$ such that $v(\varphi )< v(ThK45(\mathbf{G}))=1.$ Then $v$ is a world of the canonical model $\m_\varphi^v$ and by Lemma~\ref{equation-joint}, $e_{\varphi}^v(v,\varphi )=v(\varphi )<1.$ Thus $\not\models_{\Pi \cal G}\varphi$.
\end{proof2}


In addition, it is also easy to generalize last proof for getting completeness for deductions from finite theories as it is shown by the next theorem:

\begin{theorem} [Finite strong completeness $K45(\mathbf{G})$]
\label{JointCompleteness} For any
finite theory $T$ and formula $\varphi $ in $\mathcal{L}_{\square \Diamond }$, we have: $T\models _{\Pi\mathcal{G}}\varphi $ iff $T\vdash _{K45(\mathbf{G})}\varphi $. 
\end{theorem}

\begin{proof2}  The proof is an easy adaptation of the one of weak completeness. We will only mention the main differences. If $T \not\vdash_{K45(\mathbf{G})} \varphi$, by completeness of G\"odel logic and Lemma \ref{reduction2}, there exists a G\"odel valuation $v$  such that $ v(ThK45(\mathbf{G}))=1$, $v(T) = 1$ and $v(\varphi)<1$. Now, in order to build a canonical model, we need to take into account not only $v$ and $\varphi$ but also $T$. To do that we have to follow the very same steps as before but replacing everywhere the set $Sub(\varphi)$ of subformulas of $\varphi$ by the larger set $Sub(T, \varphi)$ of subformulas of $T 
\cup \{\varphi\}$, i.e. $Sub(T, \varphi) = \bigcup_{\psi \in T \cup \{\varphi\}} Sub(\psi)$. Let us denote by $\m^v_{T, \varphi} = \langle W^c, \pi^c, e^c \rangle$ the canonical model built accordingly, where $v \in W^c$. Note that 
there is no need of any modification in neither  Lemma \ref{normalization} nor the Truth Lemma \ref{equation-joint} (except for replacing $Sub(\varphi)$ by $Sub(T, \varphi)$ in its statement). Then the theorem follows by observing that
Lemma \ref{equation-joint} guarantees that $e^c(v, \psi) = v(\psi) = 1$ for all $\psi \in T$ while $e^c(v, \varphi) = v(\varphi) < 1$. Therefore, $T \not\models _{\Pi\mathcal{G}}\varphi $. 
\end{proof2}
%

Actually, the proofs for weak and finite strong completeness of  $K45(\mathbf{G})$ with respect to the class of simplified possibilistic models $\Pi\mathcal{G}$ easily generalize to completeness of $KD45(\mathbf{G})$, the axiomatic extension  $K45(\mathbf{G})$  with axiom $D$, with respect to the class of {\em normalized} possibilistic models $\Pi^*\mathcal{G}$. 

\begin{corollary} [Finite strong completeness $KD45(\mathbf{G})$]
For any
finite theory $T$ and formula $\varphi $ in $\mathcal{L}_{\square \Diamond }$, we have: $T\models _{\Pi^*\mathcal{G}}\varphi $ iff $T\vdash _{KD45(\mathbf{G})}\varphi $. 
\end{corollary}

\begin{proof} 
We only consider the proof of weak completeness, its extension to a proof 
of finite strong completeness can be then devised as in Theorem \ref{JointCompleteness}. Indeed, the proof only needs small adaptations to the one 
for the case of $K45(\mathbf{G})$. {To start with, for technical reasons that will become clear later, we will actually prove that $\not\vdash _{KD45(\mathbf{G})} \varphi'$ implies $\not\models _{\Pi^*\mathcal{G}} \varphi'$, where $\varphi' = \Diamond \top \to \varphi$. Note that since $\Diamond \top$ is an axiom of $KD45(\mathbf{G})$ and  is valid  in the class of frames $\Pi^*\mathcal{G}$, the former condition is indeed equivalent to prove that $\not\vdash _{KD45(\mathbf{G})} \varphi$ implies $\not\models _{\Pi^*\mathcal{G}} \varphi$. }
Next, in the definition of the set of worlds $W^v_{\varphi'}$  of the canonical model $\m^v_{\varphi'}$, we need to replace the condition $u(ThK45(\mathbf{G}))=1$ by $u(ThKD45(\mathbf{G}))=1$, i.e.  we define 
$$W^v_{\varphi'} = \{u \in \lbrack 0,1]^{Var\cup X} \mid u \sim_{\varphi'} v \mbox{ and } u(ThKD45(\mathbf{G}))=1 \},$$ 
and analogously in the condition $\bf a.$ of Lemma \ref{normalization}. Then in the proof of this lemma (item (ii) after the {\bf Claim}), one has to further check that $w(\Diamond \gtop) = 1$. Observe that $\Diamond \gtop$ is provably equivalent to $\square \gtop \to \Diamond \gtop$, but this formula is of the form $\phi \to \psi$ for some $\phi, \psi \in X$, and hence it falls under the cases already considered in the proof. Thus Lemma \ref{normalization} holds, and the same happens with Lemma \ref{equation-joint}, that holds as well without any modification. Moreover, Lemma \ref{equation-joint} allows us to 
prove that the canonical model belongs in fact to the class $\Pi^*\mathcal{G}$ of normalized possibilistic models. Indeed, by \eqref{Diam} it follows that 
$$\sup_{u^{\prime }\in W^v_{\varphi'}}\{\min(\pi^c(u^{\prime }), u^{\prime }(\psi)) \}=u(\Diamond \psi )$$ { for every $\Diamond \psi \in Sub(\varphi')$, and since $\Diamond \top \in Sub(\varphi')$ and $u \in W^v_{\varphi'}$ (and hence $u(\Diamond \gtop) = 1$), we finally have $$ 1 = u(\Diamond \gtop) = \sup_{u^{\prime }\in W^v_{\varphi'}}\{\min(\pi^c(u^{\prime }), u^{\prime }(\gtop)) \}=  \sup_{u^{\prime }\in W^v_{\varphi'}} \pi^c(u^{\prime }), $$
in other words, $\pi^c$ is normalized and thus $\m^v_{\varphi'} \in \Pi^*\mathcal{G}$. 
In summary, we have found a model $\m^v_{\varphi'} = \langle W^v_{\varphi'}, \pi^c, e_{\varphi'}^v \rangle \in \Pi^*\mathcal{G}$ and a world $v \in W^v_{\varphi'}$ such that $e_{\varphi'}^v(v, \varphi') = e_{\varphi'}^v(v, \varphi) < 1$, and therefore $\not\models _{\Pi^*\mathcal{G}} \varphi$. }
\end{proof}

%
%
%

{We would like to  finish this section by noticing that the same kind of proof for weak and finite strong completeness for $K45$ can also be easily adapted for the logic $KT45(\mathbf{G})$, that is in fact equivalent to $KT5(\mathbf{G})$, since Axiom (4) is derivable in $KT5(\mathbf{G})$. For that, we need to adapt some details of our original proof. First, we can take the same definition of the canonical model but considering $ThKT5(\mathbf{G})$ instead of  $ThK45(\mathbf{G})$. Next, we follow with Lemma \ref{normalization} where we change condition {\bf a.} by $\nu(ThKT5(\mathbf{G})) = 1$. It is easy to check that the same proof goes through except  that we need to prove now that $w(\Box \varphi \to \varphi) = w(\varphi \to \Diamond\varphi)=1$. But again this is easy to be verified,  
and the rest of the proof is working well. Then, we are are able to prove the following:

\begin{theorem} [Weak completeness $KT45(\mathbf{G})$]
\label{WeakCompletenessT45} For any formula $\varphi $ in $\mathcal{L}_{\square \Diamond }$:
$$\models _{\Pi^5\mathcal{G}}\varphi \mbox{ iff } \vdash _{KT5(\mathbf{G})}\varphi $$
\end{theorem}

\noindent where $\Pi^5\mathcal{G}$ is the class of possibilistic frames $(W, \pi)$ that validate axioms $(T_\Box)$ and $(T_\Diamond)$. However this result is not new, since it turns out that  $(W, \pi)$ validates $(T_\Box)$ and $(T_\Diamond)$ iff $\pi$ is such that $\pi(w) = 1$ for every $w \in W$. In other words, $\Pi^5\mathcal{G}$ is in fact the class of universal models,  the simplified semantics for $S5(\mathbf G)$, a result that is well-known in the literature \cite{CMRT19}. 
}

\section{Decidability} \label{sec:decidability}
So far, we have shown that $Val( \Pi{\mathcal{ G}}) = ThK45(\mathbf{G})$, i.e.\ the set of valid formulas in   $\Pi\mathcal {G}$, the class of all $\Pi G$-frames, coincides with the set of theorems of the logic $K45(\mathbf{G})$, or in other words, the logic $K45(\mathbf{G})$ is sound and complete with respect to $\Pi G$-frames. It is natural to ask whether the logic 
$K45(\mathbf{G})$ is decidable. 
Unfortunately, 
it is easy to check that under the possibilistic semantics the logic $K45(\mathbf{G})$ does not satisfy the finite model property. Indeed, consider the formula
$$\square \lnot \lnot p \rightarrow \lnot \lnot \square p  $$ 
where $p$ is a propositional variable. This formula is not valid in the model $\m = \langle \mathbb{N},\pi, e\rangle,$ where for all $n \in \mathbb{N}$, 
\begin{equation*}
\pi (n)=1 \mbox{  and  } e(p,n)=\tfrac{1}{n+1}.
\end{equation*}%
Then, for all $n \in \mathbb{N}$ we have:
$$e(\lnot \lnot p,n)=(\frac{1}{n+1} \Rightarrow 0) \Rightarrow 0=1,$$ 
 
$$e(\Box p,n )=\inf_{n \in \mathbb{N}} \{ \pi(n) \Rightarrow e(p, n)\} = \inf_{n\in \mathbb{N}}\{1{\Rightarrow} \frac{1}{n+1}\}=0,$$
$$e(\Box \neg\neg p,n )=\inf_{n \in \mathbb{N}} \{ \pi(n) \Rightarrow e(\neg\neg p, n)\} = \inf_{n\in \mathbb{N}}\{1{\Rightarrow} 1 \}=1,$$
and hence  $e(\square \lnot \lnot p \rightarrow \lnot \lnot \square p, n) 
= 0$.
However, $\Box \lnot \lnot p \rightarrow \lnot \lnot \Box p $ is valid in 
any $\Pi G$-model $\langle W,\pi,e\rangle $ where $W$ is finite.

Nevertheless, in \cite{CMRR2013}, an alternative semantics was provided for $K(\mathbf{G})$ which admits the finite model property. In this section, we are going to adapt
that semantics in order to obtain an equivalent class of $\Pi G$-models. For that, we are going to use the same strategy used in \cite{CMRR2013}, i.e. by limiting  the truth-values of modal formulae to a finite number of possibilities. According to this idea,  we propose the following adaptation of the original semantics given in  \cite{CMRR2013}.
\begin{definition}
A  ${\Pi GF}$-model is a quadruple $\m= \langle W, \pi, T, e \rangle$, where $ \langle W, \pi, e \rangle
$ is a $\Pi G$-model and $T \in \mathcal{P}_{<\omega}(\left[ 0, 1 \right])$ is a finite set of
truth values satisfying $\{0,1\} \subseteq T \subseteq [0,1]$. 
The valuation $e$ is extended to formulas using the same clauses for non-modal connectives as for $\Pi G$-models, and using the following revised clauses for modal connectives:
\begin{eqnarray*}
	e(\bo \f, x)	&	=	&	\max \left \{r \in T: r \leq \inf_{y \in W} \{\pi(y) {\Rightarrow} \ e(\f, y)\} \right \} \\
	e(\di \f, x)	&	=	&	\min \left \{r \in T: r \geq \sup_{y \in W} \{\min(\pi(y),e(\f, y))\} \right \}.
\end{eqnarray*}
A formula $\varphi$ is said to be \emph{valid} in $\m$ if $e(\varphi,x) = 1$ for all $x \in W$. We will denote by $\Pi \mathcal{GF}$ the class of all $\Pi GF$-models.
\end{definition}

Notice that now the formula $\f = \bo \neg \neg p \to \neg \neg \bo p$ has a finite
$\Pi GF$-counter-model. Indeed, consider the $\Pi GF$-model $ \m_0 = 
\langle W,\pi,T,e \rangle$ with $W = \{a\}$, $\pi(a) = 1 $, $T
= \{0,1\}$, and such that $e(p,a) = \frac{1}{2}$. Then we have: 
\begin{itemize}
\item $e(\neg p, a) = 0$, $\pi(a) \Rightarrow e(\neg \neg p, a) = 1$, 
and so $e(\bo \neg \neg p, a) = 1$.
\item Further, $e(\bo p, a) =  0$ (since $\pi(a) {\Rightarrow} \ e(p, a) = \frac{1}{2}$,
and $0$ is the next smallest element of $ T$);
\item Hence, $e(\neg \bo p, a) = 1$ and $e(\neg \neg \bo p, a) = 0$.
\item Therefore,  $1 = e(\bo \neg \neg p, a) > e( \neg \neg \bo p, a) = 
0$ and $\bo \neg \neg p \to \neg \neg \bo p$ is not valid in $\frm{M}_0$.
\end{itemize}

Next, we are going to prove that both semantics characterize the same logic. First we need the following lemmas to prove the main result.

\begin{lemma} \label{extended}
Let $\m =  \langle W, \pi, T, e \rangle$ be an $\Pi GF$-model. Given an 
order-embedding $h \colon [0,1] \to [0,1]$ satisfying $h(0) =0$ , $h(1) 
= 1$, and for any $t \in T :  h(t) = t$, consider $\wh{\m} =   \langle \wh{W}, \wh{\pi}, \wh{T}, \wh{e} \rangle $,  with $\wh{W} = W_{\m}$, $\wh{\pi}( x) = h(\pi (x))$,
$\wh{T}(x) = T(x)$, and  $\wh{e}(p,x) = h(e(p,x))$ for all  $x \in W$
and $p \in \rm{Var}$. Then, for all $\f \in \mfml$ and $x \in W$: $$\wh{e}(\f,x) = h(e(\f,x))$$
\end{lemma}
\begin{proof2} It is a special case of part (c) of Lemma 1 in~\cite{CMRR2013}.
\end{proof2}

Now, we provide the key construction of a $\Pi G$-model taking the same truth values for formulae as a given $\Pi GF$-model.

\begin{lemma}\label{rtlD45}
For any  $\Pi GF$-model $\m = \la W,\pi, T, e \ra$, there is a  $\Pi G$-model $\wh{\m}= \la \wh{W},\wh{\pi},\wh{e} \ra$ with $W \subseteq \wh{W}$,
such that $\wh{e}(\f, x) = e(\f,x)$ for all $\f \in \mfml$ and $x \in W_\m$.
\end{lemma}

\begin{proof2}
We proceed similarly to the proof of Lemma 4 in \cite{CMRR2013}. Given a $\Pi GF$-model $\m = \la W,\pi, T,e \ra$, we  construct its associated $\Pi G$-model $\wh{\m}$
directly by  taking infinitely many copies of $\m$. Consider $T = \{\alpha_1, \ldots, \alpha_n\}$ with $0 = \alpha_1 < \ldots < \alpha_n = 1$ and, using Lemma~\ref{extended}, define a family of order-embeddings $\{h_k\}_{k \in \mathbb{Z}^+}$ from $[0,1]$ into $[0,1$]
satisfying $h_k (0) = 0$ and $h_k(1) = 1$, such that
\begin{center}
\begin{tabular}{rcll}
$h_k(\alpha_i)$					&	$=$	&	$\alpha_i$ &	for all $i \leq n-1$ and $k \in \mathbb{Z}^+$\\
$h_k[(\alpha_i, \alpha_{i+1})]$	&	$=$	&	$(\alpha_i, \min(\alpha_i + \frac{1}{k}, \alpha_{i+1}))$	 &	for all $i \leq
n-1$ and even $k \in \mathbb{Z}^+$\\
$h_k[(\alpha_i, \alpha_{i+1})]$	&	$=$	&	$(\max(\alpha_{i}, \alpha_{i+1}- \frac{1}{k}), \alpha_{i+1})$	&	for all $i
\leq n-1$ and odd $k \in \mathbb{Z}^+$.\\
\end{tabular}
\end{center}

For all  $k \in \mathbb{Z}^+$, we define a  $\Pi G$-model
 $\wh{\m}_k = \langle \wh{W}_k, \wh{\pi}_k, \wh{e}_k \rangle$ such that 
each $\wh{W}_k$ is a copy of $W$ with distinct
worlds,  $ \wh{\pi}_k = h_k(\pi)$ and $\wh{e}_k(\f,x^k) = h_k(e(\f,x))$ for each copy $x^k$ of $x \in W$ and $\f \in \mfml$. We also  let
$\wh{W}_0 = W$,  $ \wh{\pi}_0 = \pi$ and $\wh{e}_0 = e$. Then $\wh{\m} = \langle \wh{W},\wh{\pi},\wh{e} \rangle$ where
\[
\wh{W} = \bigcup_{k \in \mathbb{N}} \wh{W}_k \qquad {\rm and}  \ \text{ 
for } \wh{x} \in \wh{W}_k : \qquad \wh{\pi}(\wh{x}) = \wh{\pi}_k(\wh{x}) \qquad ;\qquad
\wh{e}(p,\wh{x}) = \wh{e}_k(p,\wh{x}) .
\]
Now, It then suffices to prove that $\wh{e}(\f,x) = e(\f,x)$ for all $\f \in \mfml$ and  $x \in W$, proceeding by an induction on $\ell(\f)$.  The base case $\ell(\f) = 1$ follows directly from the definition of $\wh{e}$. For the inductive step, the cases for the
non-modal connectives follow easily using the induction hypothesis. Let us just consider the case $\f = \bo \p$, the case $\f = \di \p$ being very similar. There are two possibilities.  Suppose first that
\[
e(\bo \p,x) = \max\{r \in T : r \leq \inf \{\pi(y) \Rightarrow e(\p, y): y \in W\}\} = 1.
\]
That means for all $y \in W \  : \  \pi(y) \leq e(\p, y)$ and hence,  for 
all  $k \in \mathbb{Z}^+$ and  $\wh{y} \in \wh{W}$:  $h_k(\pi(y)) = \wh{\pi}_k ({y}^k) = \wh{\pi} (\wh{y}) \leq h_k (e(\p, y)) = \wh{e}_k(\p,y^k)= \wh{e}(\p,\wh{y})$.
It follows that
\[
\wh{e}(\bo \p, \wh{x}) = \inf\{\wh{\pi}(\wh{y}) \Rightarrow \wh{e}(\p,\wh{y}): \wh{y} \in \wh{W}\} = 1 = e(\bo \p,x).
\]
Now suppose that $e(\bo \p,x) = \alpha_i < 1$ for some $i \leq m-1$. Then  $\pi(z) \Rightarrow e(\p,z) \geq \alpha_i$ for all $z \in W$, and thus, $(\star)$, $\wh{\pi} (\wh{z}) \Rightarrow \wh{e}(\p,\wh{z}) \geq \alpha_{i}$ for all $\wh{z}
\in \wh{W}$, by construction using the order-embeddings $\{h_k\}_{k \in \mathbb{Z}^+}$.

There are two subcases. First, suppose that there is at least one $y \in W$ such that $\pi(y) \Rightarrow e(\p,y) =
\alpha_i$; call it $y_0$. This means that $\pi(y_0) > e(\p,y_0) = \alpha_i$ and for all $k \in \mathbb{Z}^+$, $\wh{e}(
\p, \wh{y}_0) = \wh{e}_k(\p,{y}_0^k) =  h_k(e(\p,y_0))
= h_k(\alpha_i)= \alpha_i$. Since $\pi(y_0) > \alpha_i$, also for all 
$k \in \mathbb{Z}^+$, $\wh{\pi}(\wh{y}_0) = \wh{\pi}_k({y}_0^k) =
h_k(\pi(y_0)) > \alpha_i = \wh{e}(\p, \wh{y}_0)$, and hence, using $(\star)$,
\[
\wh{e}(\bo \p,\wh{x})=\inf\{\wh{\pi} (\wh{z}) \Rightarrow \wh{e}(\p,\wh{z}): \wh{z} \in \wh{W}\} = \alpha_i = e(\bo \p, x).
\]
Now suppose that $\pi(y) \Rightarrow e(\p,y) > \alpha_i$ for all $y \in W$. Since $e(\bo \p,x) = \max\{r \in T
: r \leq \inf \{\pi(y) \Rightarrow e(\p, y): y \in W\}\} = \alpha_i$, there is at least one $y \in W$ such that $\pi(y) \Rightarrow e(\p,y) \in (\alpha_i, \alpha_{i+1})$; call it $y_0$. Then, by construction, for any $\epsilon > 0$ there is a $k \in \mathbb{Z}^+$ such that $\wh{\pi} ({y}^k_0) \Rightarrow \wh{e}(\p, {y}^k_0) \in (\alpha_i, \alpha_i + \epsilon)$. 
Using $(\star)$, this ensures that\\

\qquad$\wh{e}(\bo \p,\wh{x})= \inf\{\wh{\pi}( z) \Rightarrow \wh{e}(\p,z): z \in \wh{W}\} = \alpha_i = e(\bo\p,x)$.

\end{proof2}
As an immediate consequence of Lemma \ref{rtlD45}, we have the next corollary. 

\begin{corollary} \label{Decib1} $Val( \Pi{\mathcal{ G}}) \subseteq  Val( 
\Pi{\mathcal{ GF}})$. 
\end{corollary}

The next lemma paves the way to prove in Theorem \ref{Decib2} that the logic $K45(\mathbf{G})$ has the finite model property and offers a bound
on the complexity. We call a set $\Sigma \subseteq \mfml$ a \emph{fragment} if it is closed under subformulas.

\begin{lemma} \label{ltrD45}
{Let $\Si \subseteq \mfml$ be a finite fragment.} Then, for any  $\Pi G$-model $\m$, there is a finite  $\Pi GF$-model $\wh{\m}$ with $\wh{W} 
\subseteq W$, such that
$\wh{e}(\f, x) = e(\f,x)$ for all $\f \in \Si$ and $x \in \wh{W}$. Moreover, $|\wh{W}|
+|\wh{T}| \leq 2|\Si|$.
\end{lemma}
\begin{proof2}
{Let $\Si\subseteq \mfml$ be a finite fragment}, $\m = \langle W, \pi, e \rangle$ a  $\Pi G$-model.  First, define  $\Si_\bo$ as the set of 
all box-formulas in $\Si$, $\Si_\di$ as the set of all diamond-formulas in $\Si$, and $\Si_{
\rm{Var}}$  as the set of all variables in $\Si$. Let us also define $e_{x}[\De] = \{e(\f,x): \f \in \De\}$ for any $x
\in W$ and $\De \subseteq \mfml$.  In addition,  let $e_{x}[\Si_\bo \cup \Si_\di] \cup  \{0,1 \}=\{\alpha_1, \ldots, \alpha_n\}$ with $0 = \alpha_1 < \ldots < \alpha_n = 1$.

Next, we choose a finite number of $y \in W$. For each $\bo \p \in \Si_\bo$ such that $e(\bo \p,x) = \alpha_i < 1$, choose a $y = y_{\bo \p}
 \in W$ such that $e (\p,y_{\bo \p}) < \alpha_{i+1}$, and for each $\di \p \in \Si _\di$, such that $e(\di \p,x) =
\alpha_i > 0$, choose a $y = y_{\di \p} \in W$ such that $e( \p,y_{\di \p}) > \alpha_{i-1}$. Then let $\wh{W} = \{x
\} \cup \{y_{\bo \p} \in W: \bo \p \in \Si_\bo\} \cup \{y_{\di \p} \in W: 
\di \p \in \Si_\di\}$. Clearly $\wh{W}
\subseteq W$ is finite. We define $\wh{\m} = \langle \wh{W}, \wh{\pi}, \wh{T}, \wh{e} \rangle$ where $\wh{T} = e_{x}[\Si_\bo
\cup \Si_\di] \cup \{0,1\}$, and both $\wh{\pi}$ and $\wh{e}$ are equal to $\pi$ and $e$  restricted to $\wh{W}$, respectively. It then follows by 
induction on $\ell(\f)$ that $\wh{e}(\f, x) = e(\f,x)$ for all $x \in \wh{W}$ and $\f \in\Si$.  The base case follows from the definition of $\wh{e}$.  For the inductive step, let $\f \in \Si$ be of the form $\f = \bo \p$ (the non-modal cases follow directly, using the induction hypothesis). We need to consider the same two cases it was considered in previous 
Lemma \ref{rtlD45}. First, note that always it is true that:
$$ \inf\{ \pi(y) \Rightarrow e(\p,y): {y \in W} \} \leq  \inf\{ \wh{\pi}(\wh{y}) \Rightarrow \wh{e}(\p,\wh{y}): \wh{y} \in \wh{W}\}$$
because of  $\wh{W} \subseteq W$. Now suppose that $1 = e(\bo \p, x)$. Then, because $1 \in \wh{T}$,
\begin{align*}
	1 = e(\bo \p, x) &= \inf\{ \pi(y) \Rightarrow e(\p,y): {y \in W} \} \\
		&\leq \inf\{ \wh{\pi}(\wh{y}) \Rightarrow \wh{e}(\p,\wh{y}): \wh{y} \in \wh{W}\} \\
		&\leq \max\{ r \in \wh{T} : r \leq  \inf\{ \wh{\pi}(\wh{y}) \Rightarrow \wh{e}(\p,\wh{y}): \wh{y} \in \wh{W}\}\} \\
		&= \wh{e}(\bo \p, x).
\end{align*}

For the second case, $e(\bo \p, x) = \inf\{ \pi(y) \Rightarrow e(\p,y) : y \in W\}  = \alpha_i < 1$ for some $i \in \{1, \ldots, n-1\}$. According to our choise, there exists a  $ y_{\bo \p} \in \wh{W}$ such that $ \alpha_i \leq \wh{\pi}(y_{\bo \p}) \Rightarrow \wh{e}(\p, y_{\bo \p}) <  \alpha_{i+1}$. Hence, $$ \alpha_i \leq  \inf\{ \wh{\pi}(\wh{y}) \Rightarrow \wh{e}(\p,\wh{y}): \wh{y} \in \wh{W}\}\} \leq \wh{\pi}(y_{\bo \p}) \Rightarrow \wh{e}(\p, y_{\bo \p}) <  \alpha_{i+1}$$
Thus
$$ \wh{e}(\bo \p, x)= \max\{ r \in \wh{T} : r \leq  \inf\{ \wh{\pi}(\wh{y}) \Rightarrow \wh{e}(\p,\wh{y}): \wh{y} \in \wh{W}\}\} =  \alpha_i$$

The diamond-case case follows similarly to the box-case and is therefore omitted.

Finally, we note that by the construction of $\wh{\m}$, $|\wh{W}| \leq |\Si_\bo \cup \Si_\di| + 1 \leq |\Si|$ and $|\wh{T}| \leq |\Si_\bo \cup \Si 
_\di| + 2 \leq |\Si|$, and therefore $|\wh{W}| + |\wh{T}| \leq 2 |\Si|$.

\end{proof2}

As a direct consequence of the above lemma  we have the converse inclusion of Corollary \ref{Decib1}.

\begin{corollary}  $Val( \Pi{\mathcal{ GF}}) \subseteq Val( \Pi{\mathcal{ 
G}})   $. 
\end{corollary}

Finally, we can state the main result of this section. 

\begin{theorem}[main] \label{Decib2} For each $\f \in \mfml$:
$\mdl{\Pi \mathcal{G}} \f$ iff $\f$ is valid in all {(finite)} $\Pi GF$-models $\m = \langle W,\pi,T,e \rangle$ satisfying
					$|W|+|T| \leq 2(\ell(\f)+2)$.
\end{theorem}

{ As a direct consequence we get the decidability of $K45({\bf G})$. 

\begin{corollary} The logic $K45({\bf G})$ is decidable. 
\end{corollary}

Similarly, one can prove that the logic $KD45({\bf G})$ is decidable as well. 

\begin{corollary} The logic $KD45({\bf G})$ is decidable. 
\end{corollary}
}


\section{Conclusions} \label{sec:Disc-concl}
In this paper we have studied the logic over G\"odel fuzzy logic arising from many-valued G\"odel Kripke models with possibilistic semantics, and have shown that it actually corresponds to a simplified semantics for the 
logic $K45({\bf G})$, the extension of Caicedo and Rodriguez's bi-modal G\"odel logic with many-valued versions of the well-known modal axioms  $4$ and $5$. We have also considered the extension with the axiom $D$, the logic $KD45({\bf G})$, and have shown to be captured by normalised possibilistic G\"odel Kripke models. In this way, we have obtained many-valued G\"odel generalizations of the results reported by Pietruszczak in \cite{Pietrus09} about simplified semantics for several classical modal logics. We have also shown the decidability of those logics.

It is worth noticing that the truth-value of a formula $\Diamond \varphi$ 
in a possibilistic Kripke model is indeed a proper generalization of the possibility measure of $\varphi$ when $\varphi$ is a classical proposition, however the semantics of $\Box \varphi$ is not. This is due to the fact that the negation in G\"odel logic is not involutive. Therefore, a first open problem we leave for further research is to consider to extension of the logic $K45({\bf G})$ with an involutive negation and investigate its possibilistic semantics. {Finally, we would like to study the connection between our simplified possiblistic semantics and the pseudomonadic algebras proposed in \cite{BCR2019}}.

\paragraph{Acknowledgments} Rodriguez acknowledges  partial support  of Argentinean projects: PICT-2019-2019-00882, UBA-CyT-20020190100021BA and PIP 112-2015-0100412 CO.
Tuyt is supported by Swiss National Science Foundation (SNF) grant 200021\_184693. Esteva and Godo acknowledge  partial support  by the Spanish project 
PID2019-111544GB-C21.

\section*{Appendix}

%


\noindent \textbf{Claim from Lemma \ref{normalization}}.  {\em Induction hypothesis (IH):  if Properties 1--4 are satisfied by formulas with complexity at most $n$ then these properties keep holding for formulae with complexity at most $n+1$.}
\begin{proof2}
This case will become important when we try to prove that $w$ satisfies the axioms of $K45(\mathbf{G})$. \label{extraproperty}
%

We consider different cases according to the main operator being either $\wedge$ or $\to$.
First, we prove Properties 1 and 2, that involve only one formula:
  \begin{description}
    \item[Case $\psi= \psi_1 \wedge \psi_2$.]  \ \vspace{-3mm} \\
    \begin{enumerate}
      \item Let $\nu(\psi_1 \wedge \psi_2)=1$.  Then $\nu(\psi_1) = \nu(\psi_2)=1$. In such a case, from the inductive hypothesis, we conclude that $w(\psi_1) \geq \delta +\varepsilon$ and $w(\psi_2)\geq \delta +\varepsilon$. Therefore, $w(\psi_1 \wedge \psi_2) \geq \delta +\varepsilon$.
      \item Let $\nu(\psi_1 \wedge \psi_2)< 1$. Then, without loss of generality, we can assume that $\nu(\psi_1 \wedge \psi_2)=\nu(\psi_1)\leq \nu(\psi_2)$. In the case that $\nu(\psi_1) < 1$, we can apply the inductive hypothesis and conclude that $w(\psi_1) < \delta + \varepsilon$ and $w(\psi_1) \leq w(\psi_2)$, according to Properties 2 and 3, respectively. 
Therefore, we obtain $w(\psi_1 \wedge \psi_2) = w(\psi_1) < \delta+\varepsilon$.
    \end{enumerate}
    \item[Case $\psi= \psi_1 \to \psi_2$.] \ \vspace{-3mm} \\
    \begin{enumerate}
      \item Let $\nu(\psi_1 \to \psi_2)=1$. Then $\nu(\psi_1) \leq \nu(\psi_2)$. First, we consider the case $\nu(\psi_1)<1$ where we can apply Property 3 by the inductive hypothesis, obtaining that $w(\psi_1) \leq w(\psi_2)$ and, hence, $w(\psi_1 \to \psi_2) =1 \geq \delta+\varepsilon$. 
Now, we turn to the case $\nu(\psi_1) = 1 = \nu(\psi_2)$. In this situation, by applying Property 1 on $\psi_2$, we conclude that $w(\psi_1 \to \psi_2) \geq w(\psi_2) \geq \delta + \varepsilon$.
      \item Let $\nu(\psi_1 \to \psi_2)< 1$. Then $\nu(\psi_1) > \nu(\psi_2)$. By inductive hypothesis, by applying Property 4, we have that  $w(\psi_1) > w(\psi_2)$, and hence $w(\psi_1 \to \psi_2) = w(\psi_2)$, but since  $\nu(\psi_2)< 1$, then by Property 2, $w(\psi_2) < \delta + \varepsilon$.
    \end{enumerate}
  \end{description}

  Next we prove Properties 3 and 4, that involve two formulas. We proceed 
by considering the four possible combinations.
  \begin{description}
    \item[Case $\psi=\psi_1 \wedge \psi_2$ and $\phi=\phi_1 \wedge \phi_2$.] In this case the proof is easy because, without loss of generality, we can assume that $\nu(\psi_1 \wedge \psi_2) = \nu(\psi_1)$ and $\nu(\phi_1 \wedge \phi_2) = \nu(\phi_1)$.
    \begin{description}
      \item[$3.$] Let $1 \neq \nu(\psi_1 \wedge \psi_2)\leq \nu(\phi_1 \wedge \phi_2)$. According to our assumption, this means $1 \neq \nu(\psi_1) \leq \nu(\phi_1)$, and by applying induction hypothesis we obtain $w(\psi_1) \leq w(\phi_1)$ and $w(\psi_1) < \delta + \varepsilon$. Now, we need to consider two additional cases:
          \begin{description}
            \item[$a.$] If $\nu(\psi_2) = 1$ then, by IH, $w(\psi_2) \geq \delta + \varepsilon$. Thus, we obtain $w(\psi_1\wedge \psi_2) = w(\psi_1) < \delta + \varepsilon$. In addition, if $\nu(\phi_1)<1$ then, by IH (Property 3), $w(\phi_1) \leq w(\phi_2)$ and $w(\psi_1 \wedge \psi_2)\leq w(\phi_1 \wedge \phi_2)$. On the contrary, if $\nu(\phi_1)=1$ then, $\nu(\phi_2)=1$ and $w(\phi_1 \wedge \phi_2) \geq \delta + \varepsilon > w(\psi_1 \wedge \psi_2)$.
            \item[$b.$] If $\nu(\psi_2) < 1$ then, by IH on Properties 2 and 3, $w(\psi_1) \leq w(\psi_2) < \delta + \varepsilon$ which implies $w(\psi_1)= w(\psi_1 \wedge \psi_2) \leq w(\phi_1)$. The proof of this case is completed by reproducing the both alternatives when $\nu(\phi_1)< 1$ and $\nu(\phi_1)=1$, which were analyzed in the previous item.
          \end{description}
      \item[$4.$] Let $\nu(\psi_1 \wedge \psi_2) < \nu(\phi_1 \wedge \phi_2)$. Again, this means $\nu(\psi_1) < \nu(\phi_1)$. By IH, we obtain $w(\psi_1) < w(\phi_1)$ and $w(\psi_1) < \delta + \varepsilon$ (Properties 4 
and 2, respectively). The rest of the proof runs as before.
    \end{description}
    \item[Case $\psi=\psi_1 \to \psi_2$ and $\phi=\phi_1 \wedge \phi_2$.] Again, without loss of generality, we can assume that $\nu(\phi_1 \wedge \phi_2) = \nu(\phi_1)$.
    \begin{description}
      \item[$3.$] Let $1 \neq \nu(\psi_1 \to \psi_2)\leq \nu(\phi_1 \wedge \phi_2)$. Since $\nu(\psi_1 \to \psi_2)< 1$, we conclude that $\nu(\psi_1) > \nu(\psi_2)$ and $1 \neq \nu(\psi_2) \leq \nu(\phi_1)$. By IH, we know that $w(\psi_2) \leq w(\phi_1)$ and $w(\psi_1)> w(\psi_2)$. In addition, if $1\neq \nu(\phi_1) \leq \nu(\phi_2)$ then, by IH, we have $w(\psi_1 \to \psi_2) = w(\psi_2) \leq w(\phi_1) = w(\phi_1 \wedge \phi_2)$. On the contrary, if $\nu(\phi_1) = 1 = \nu(\phi_2)$ then $w(\phi_1 \wedge \phi_2) \geq \delta + \varepsilon > w(\psi_2) = w(\psi_1 \to \psi_2)$.
      \item[$4.$] Let $\nu(\psi_1 \to \psi_2) < \nu(\phi_1 \wedge \phi_2)$. 
      As in the previous proof, we can conclude that $\nu(\psi_1) > \nu(\psi_2)$ and so by Property 4, $w(\psi_1) > w(\psi_2)$, i.e.\ $w(\psi_1 \to \psi_2) = w(\psi_2)$. Moreover, as $\nu(\psi_2) < \nu(\phi_1)$, $w(\psi_2) < w(\phi_1)$ again by Property 4. 
      If $1\neq \nu(\phi_1) \leq \nu(\phi_2)$ then, by IH, $w(\phi_1) = w(\phi_1 \wedge \phi_2)$. On the contrary, if $\nu(\phi_1) = 1 = \nu(\phi_2)$ then, by IH, $w(\phi_1 \wedge \phi_2) \geq \delta + \varepsilon > w(\psi_2) = w(\psi_1 \to \psi_2)$.
    \end{description}
    \item[Case $\psi=\psi_1 \wedge \psi_2$ and $\phi=\phi_1 \to \phi_2$.] Again, without loss of generality, we can assume that $\nu(\psi_1 \wedge \psi_2) = \nu(\psi_1)$.
    \begin{description}
      \item[$3.$] Let $1 \neq \nu(\psi_1 \wedge \psi_2)\leq \nu(\phi_1 \to \phi_2)$. According to our assumption, we know $\nu(\psi_1 \wedge \psi_2)=\nu(\psi_1)<1$. Then, by IH, we obtain $w(\psi_1 \wedge \psi_2)=w(\psi_1)< \delta + \varepsilon$. If $\nu(\phi_1 \to \phi_2)<1$ then $ \nu(\phi_1) > \nu(\phi_2) \geq \nu(\psi_1)$ and, by IH, we know $w(\phi_1 \to 
\phi_2)= w(\phi_2)\geq w(\psi_1)= w(\psi_1 \wedge \psi_2)$. On the contrary, $\nu(\phi_1 \to \phi_2)=1$ implies $\nu(\phi_1) \leq \nu(\phi_2)$. In this case, if $\nu(\phi_1) < 1$ then, by IH $w(\phi_1) \leq w(\phi_2)$ which implies $w(\phi_1 \to \phi_2)=1$. If $\nu(\phi_1) = 1 = \nu(\phi_2)$ then, by IH, $w(\phi_1 \to \phi_2) \geq w(\phi_2) \geq \delta + \varepsilon > w(\psi_1) = w(\psi_1 \wedge \psi_2)$.
      \item[$4.$] Let $\nu(\psi_1 \wedge \psi_2) < \nu(\phi_1 \to \phi_2)$. In this case, we can proceed analogously to the previous proofs. For this reason, we leave to the reader to complete the details.
    \end{description}
    \item[Case $\psi=\psi_1 \to \psi_2$ and $\phi=\phi_1 \to \phi_2$.] \ \vspace{-3mm} \\
    \begin{description}
      \item[$3.$] Let $1 \neq \nu(\psi_1 \to \psi_2)\leq \nu(\phi_1 \to \phi_2)$. As $\nu(\psi_1 \to \psi_2) \neq 1$, we conclude $\nu(\psi_1 \to \psi_2) = \nu(\psi_2) < 1$. In addition, if $\nu(\phi_1 \to \phi_2)< 1$ then $\nu(\phi_1 \to \phi_2)= \nu(\phi_2)$. Thus, if the truth values of both implications are less than one, we have $1 \neq \nu(\psi_2) \leq \nu(\phi_2)$. By IH, we know $w(\psi_2) \leq w(\phi_2)$, $w(\psi_1) > w(\psi_2)$ and $w(\phi_1) > w(\phi_2)$. Therefore, $w(\psi_1 \to \psi_2) = w(\psi_2) \leq w(\phi_2) = w(\phi_1 \to \phi_2)$.
          On the contrary, if $\nu(\phi_1 \to \phi_2)= 1$ then $\nu(\phi_1 \leq \nu(\phi_2)$. If $\nu(\phi_1)<1$ then, by IH, $w(\phi_1 \to \phi_2) =1$.
          In the opposite case, if $\nu(\phi_1)=1=\nu(\phi_2)$ then, by IH, $w(\phi_1 \to \phi_2) \geq w(\phi_2) \geq \delta + \varepsilon > w(\psi_2) = w(\psi_1 \to \psi_2)$.
      \item[$4.$] Let $\nu(\psi_1 \to \psi_2) < \nu(\phi_1 \to \phi_2)$. If both $\nu(\psi_1 \to \psi_2) < 1 $ and $\nu(\phi_1 \to \phi_2) < 1$ then we can proceed as in the previous proof and we can apply IH concluding 
that $w(\psi_1 \to \psi_2)=w(\psi_2) < w(\phi_2)=w(\phi_1 \to \phi_2)$. On the contrary, if $\nu(\phi_1 \to \phi_2) = 1$ then we need to consider two alternatives. First, $1\neq\nu(\phi_1)\leq \nu(\phi_2)$ implies 
by IH that $w(\psi_1 \to \psi_2)=w(\psi_2) < \delta + \varepsilon < 1 = 
w(\phi_1 \to \phi_2)$. Second, if $\nu(\phi_1)=1=\nu(\phi_2)$ then, by IH, $w(\phi_1 \to \phi_2) \geq w(\phi_2) \geq \delta + \varepsilon > w(\psi_2) = w(\psi_1 \to \psi_2)$.
    \end{description}
  \end{description}

This finishes the proof of the claim. \end{proof2}

%

%

\medskip

\noindent \textbf{Claim 1 from Lemma \ref{equation-joint}}. \emph{If $u(\Box \psi )=\alpha <1$, for every $\varepsilon > 0$, there exists a valuation $w\in W^c$ such that $\pi^c(w) > w(\psi)$ and $w(\psi ) < \alpha + \varepsilon$, and thus $\pi^c(w) \Rightarrow w(\psi ) = w(\psi ) < \alpha + \varepsilon$}.\medskip

\begin{proof2} 

The proof is achieved in two stages:
\begin{itemize}
\item first producing a valuation $ \nu \in W$ satisfying $\nu(\psi)<1$ and preserving the relative ordering conditions the values $w(\theta )$ must
satisfy, conditions which may be coded by a theory $\Gamma_{\psi,u}$;

\item and then moving the values $\nu(\theta )$, for $\theta \in Sub(\varphi)$,  to the correct
valuation $w$ by composing $\nu$ with an increasing function of [0,1], using Lemma~\ref{normalization}.
\end{itemize}
Assume $u(\Box \psi)=\alpha <1$, and define the following set of formulas (all formulas involved
ranging in\ $\mathcal{L}_{\square \Diamond }(Var)$):
\begin{equation*}
\begin{array}{ll}
\Gamma_{\psi,u}= & \{\chi : \chi \in X \mbox{ and } u(\chi) > \alpha \} 
\\
& \cup \{\lambda \rightarrow \theta  : \lambda \in \Delta_\varphi \mbox{ and } u(\lambda) \leq u(\Box \theta) \} \\
& \cup \{(\theta \rightarrow \lambda)\rightarrow \lambda : \lambda \in \Delta_\varphi \mbox{ and }  u(\lambda)< u(\Box \theta)< 1 \} \\
& \cup \{\theta   \rightarrow \lambda : \lambda \in \Delta_\varphi \mbox{ 
and } u(\Diamond \theta) \leq u(\lambda) \} \\
& \cup \{(\lambda \rightarrow \theta)\rightarrow \theta : \lambda \in \Delta_\varphi \mbox{ and }  u(\Diamond \theta)< u(\lambda)< 1 \} \\
& \cup \{(\chi_1 \rightarrow \chi_2) \rightarrow \chi_2 : \chi_1 , \chi_2 
\in X \mbox{ and } u(\chi_2) < u(\chi_1) \leq \alpha \} \\
\end{array}%
\end{equation*}%
Then we can check that $u(\square \xi ) > \alpha $ for each $\xi \in \Gamma_{\psi,u}$. Indeed, first recall that, by $U_\Box$ and $U_\Diamond$ of Proposition \ref{simplif},  for any $\lambda \in X$ (in particular $\Delta_{\varphi}$) we have $u(\di \lambda) \leq {u(\lambda)} \leq u(\bo \lambda)$ in $\lgc{K45(\alg{G})}$. 
We 
analyse case by case. For the first set of formulas,  it is clear by construction. For the second set,  we have
$u(\square (\lambda \rightarrow \theta )) \geq u(\Diamond \lambda
\rightarrow \square \theta)= u(\Diamond \lambda) \Rightarrow u(\Box \theta) \geq u(\lambda) \Rightarrow u(\Box \theta) = 1$,  by $FS2$. For the 
third, by FS2 and P, we have
\begin{align*}
u(\square ((\theta \rightarrow \lambda)\rightarrow \lambda)) \geq u(\Diamond(\theta \to \lambda) \to \Box \lambda) &\geq
u((\Box \theta  \to \Diamond \lambda) \to \Box \lambda) \\
	&\geq u((\bo \theta \to \lambda) \to \bo \lambda) \\
	&= {u(\lambda) \to u(\bo \lambda)} = 1,
\end{align*}
since $u(\di \lambda) \leq u(\lambda) \leq u(\bo \lambda)$ and $u(\lambda) < u(\bo \theta)$. The fourth and fifth  cases are very similar 
to the second and third ones. {As for the sixth case, by (T4), either $u(\Box \chi_1 \to \Diamond \chi_2) = 1$  or  $u(\Box((\chi_1 \to \chi_2) \to \chi_2)) = 1$. 
According to $(U_\Box)$, $(U_\Diamond)$, $(5_\Box)$ and $(5_\Diamond)$, we have that, for any $\varphi$, $\Diamond\Box \varphi \to \Box \varphi$ and $ \Diamond \varphi \to \Box\Diamond \varphi$ are theorems of $K45(G)$,  and thus 
 $u(\Box \chi_1) \geq u(\chi_1)$ and $u(\Diamond \chi_2) \leq u(\chi_2)$, from where it follows that
$$u(\Box \chi_1 \to \Diamond \chi_2) = u(\Box \chi_1) \to u(\Diamond \chi_2) \leq u(\chi_1) \to u(\chi_2) = u(\chi_2) < 1.$$
Therefore, it holds that $u(\Box((\chi_1 \to \chi_2) \to \chi_2)) = 1$.}  
%
%
%

The fact that  $u(\square \xi ) > \alpha $ for each $\xi \in \Gamma_{\psi,u}$ implies
\begin{equation*}
\Gamma_{\psi,u} \not\vdash _{K45(\mathbf{G})}\psi,
\end{equation*}
otherwise there would exist $\xi _{1},\ldots ,\xi _{k}\in \Gamma_{\psi,u}$
such that $\xi _{1},\ldots ,\xi _{k} \vdash _{K45(\mathbf{G})} \psi.$ In such a case, we would have $\Box \xi _{1},\ldots ,\Box \xi _{k} \vdash _{K45(\mathbf{G})}\Box \psi $ by $Nec$
and $K_{\square }$. Then $\Box \xi _{1},\ldots ,\Box
\xi _{k},ThK45(\mathbf{G})\vdash _G\Box \psi$ by Lemma \ref{reduction}, and thus by Proposition \ref{ordersoundness} (i), and recalling that $u(ThK45(\mathbf{G}))=1$, we would have
\begin{equation*}
\alpha < \inf u(\{\Box \xi _{1},\ldots ,\Box \xi _{k}\}\cup ThK45(\mathbf{G}))\leq u(\square \psi )=\alpha ,
\end{equation*}
a contradiction. Therefore, by Proposition \ref{ordersoundness} (ii) there
exists a valuation $\nu:Var\cup X\mapsto \lbrack 0,1]$ such that $\nu(\Gamma_{\psi,u} \cup ThK45(\mathbf{G}))=1$ and $%
\nu(\psi )<1$. This implies the following relations between $u$ and $\nu,$
that we list for further use. Given $\lambda, \lambda_1, \lambda_2 \in \Delta_\varphi, \  \theta \in \mathcal{L}_{\square \Diamond }(Var)$ and $\chi, \chi_1, \chi_2 \in X$, we have: \label{proplambda}\medskip
\begin{description}

\item[\#\textbf{1}.] If $u(\chi)>\alpha $ then $\chi \in \Gamma_{\psi,u}$, and hence $\nu(\chi)=1$.

\item[\#\textbf{2}.] If $u(\lambda)\leq u(\square \theta)$  then $\nu(\lambda)\leq \nu(\theta )$ (since then $\lambda \rightarrow \theta \in \Gamma_{\psi,u})$. In particular, if $\lambda_1,\lambda_2 \in \Delta_\varphi$ and $u(\lambda_1)\leq u(\Box\lambda_2)$ then $\nu(\lambda_1)\leq \nu(\lambda_2)$. Furthermore, if $\Box \theta \in \Delta_\varphi$ then $\nu(\Box \theta) \leq \nu(\theta)$. That means, taking $\theta = 
\psi$,  $\nu(\Box \psi) \leq \nu(\psi) < 1$.

\item[\#\textbf{3}.] If $u(\lambda)< u(\square \theta)< 1$
then $\nu(\lambda)< \nu(\theta)$ or $\nu(\lambda)=1$ (since then $(\theta \rightarrow \lambda)\rightarrow \lambda)\in
\Gamma_{\psi,u}$). In particular, if $\lambda_1,\lambda_2 \in \Delta_\varphi$, $u(\lambda_1) < u(\lambda_2)$ and $u(\lambda_2) \leq u(\Box \psi)= 
\alpha $ then $\nu(\lambda_1) < \nu(\psi) < 1$ and thus $\nu(\lambda_1) < 
\nu(\lambda_2)$. This means that $\nu$ preserves in a strict sense the order values by $u$ of the formulas $\lambda \in \Delta_\varphi$ such that $u(\lambda) \leq \alpha$.
\item[\#\textbf{4}.] If $ u(\Diamond \theta) \leq u(\lambda)$ then $\nu(\theta) \leq \nu(\lambda)$ (because $\theta \to \lambda \in \Gamma _{\psi ,u}$). In particular, if $\Diamond \theta \in \Delta_\varphi$ then $\nu(\theta) \leq \nu(\Diamond \theta)$.

\item[\#\textbf{5}.] If  $ u(\Diamond \theta)< u(\lambda)< 1$ then $\nu(\theta) < \nu(\lambda)$ or $\nu(\theta) =1$. In particular, if $\lambda_1, \lambda_2 \in \Delta_\varphi$ and $u(\di \lambda_1) \leq u(\lambda_1)< u(\lambda_2)\leq \alpha = u(\Box \psi)$ then $\nu(\lambda_1)< \nu(\lambda_2)$. Furthermore, if $u(\lambda_2)>0$ then $\nu(\lambda_2)>0$ (taking $\lambda_1:=\Diamond\gbot$ since $u(\gbot )=u(\Diamond \gbot)=0$).
\item[\#\textbf{6}.] If $ u(\chi_2) < u(\chi_1) \leq \alpha$ then $\nu(\chi_2) < \nu(\chi_1) < 1$. Indeed, note that, by construction, $\nu(\psi) <1$, that implies, by \#\textbf{2}, $\nu(\Box \psi) < 1$. In addition, we know $\Box \psi \in \Delta_\phi$ and $u(\Box \psi) = \alpha$. Therefore, according to \#\textbf{4}, if $\chi_1 \in X$ and $u(\di \chi_1) \leq u(\chi_1) \leq u(\bo \psi) = u(\bo \bo \psi)$, then $\nu(\chi_1) \leq \nu(\bo \bo \psi) = \nu(\bo \psi) < 1$.
Finally, since $(\chi_1 \to \chi_2) \to \chi_2 \in \Gamma_{\psi ,u}$ we have $\nu(\chi_2) < \nu(\chi_1) < 1$.
\item[\#\textbf{7}.] {If $u(\lambda_1)\leq  u(\lambda_2)$ then $\nu(\lambda_1)\leq  \nu(\lambda_2)$. Note that this is a particular case of \#\textbf{2}.  Certainly,  if $\lambda_1, \lambda_2 \in \Delta_\varphi$, then by Lemma \ref{rmk} we have  
$u(\lambda_2) \leq  u(\Box\lambda_2)$, and, hence, by \#\textbf{2}, $\nu(\lambda_1)\leq  \nu(\lambda_2)$.}
\end{description}

According to the properties \#1--\#7, it is clear that $\nu$ satisfies the conditions {\bf a}-{\bf d} of Lemma \ref{normalization}. Indeed, from \#1 and \#2 it follows that
$$\alpha =  \max \{ u(\lambda) : \nu(\lambda) < 1 \mbox{ and } \lambda \in \Delta_\varphi\},$$
because, by construction of $\Gamma_{\psi,u}$ and by \#1, $\forall \lambda \in \Delta_\varphi$, if $u(\lambda) > \alpha$  then $\nu(\lambda) =1$ 
and by \#\textbf{2}, $\nu(\Box \psi) < 1$. Therefore $\alpha$ coincides with the value $\delta$ defined in  Lemma \ref{normalization}.

Consequently, for all $\varepsilon > 0$ such that $ \alpha +\varepsilon<\beta$, where $\beta$ is defined as in Lemma~\ref{normalization}, 
there exists a valuation $w \in W^{v}$ such that $w(\psi) < \alpha +\varepsilon$.
Then in order to finish our proof, it remains to show that:
\begin{equation}\label{A-bound}
  \pi^c(w)= \inf_{\lambda \in sub(\varphi)} \min(v(\Box \lambda) \Rightarrow w(\lambda),  w(\lambda) \Rightarrow v(\Diamond \lambda)) >w(\psi)
\end{equation}
To do so, we will prove that, for any $\lambda \in sub(\varphi)$, both implications in (3) are  greater than or equal to $\alpha + \varepsilon$.\footnote{Remember that $u \sim_\varphi v \sim_\varphi w$.}
First we prove it for the first implication by cases:
\begin{description}
\item[-] If $\ v(\Box \lambda ) \leq \alpha < 1$ then $v(\Box \lambda) \Rightarrow w(\lambda) = 1$. Indeed, first of all, by   \#2,
from $u (\Box \lambda) = v(\Box \lambda ) \leq \alpha = u(\Box \psi)$ 
it follows $\nu(\Box \lambda) \leq \nu(\psi) < 1$. Now, since $u(\Box\lambda) \leq u(\Box\lambda)$, by \#2, we have $1 \neq \nu(\Box \lambda) \leq 
\nu(\lambda)$, and by 3 of Lemma \ref{normalization} we have $ v(\Box\lambda) = w(\Box \lambda) \leq w(\lambda)$. Then $v(\Box \lambda) \Rightarrow w(\lambda) =1$.

\item[-] If $v(\Box \lambda ) > \alpha$  then by \#1 and \#2, $1=\nu(\Box \lambda)\leq\nu(\lambda)$. Therefore, by 1 of Lemma \ref{normalization}, $w(\lambda)\geq \alpha + \varepsilon$ which implies $v(\Box\lambda) \Rightarrow w(\lambda) > \alpha$.
\end{description}

\noindent For the second implication we also consider  two cases:
\begin{description}
\item[-] If $v(\Diamond \lambda) = u(\Diamond \lambda) > \alpha$ then $\nu(\di \lambda) =1$ by \#1, so $w(\lambda) \geq w(\di \lambda) \geq \alpha + \varepsilon$.
it is obvious that $w(\lambda) \Rightarrow v(\Diamond \lambda) > \alpha$ because by Lemma \ref{normalization} we know $w(\lambda) \geq \alpha + \varepsilon$.

\item[-]  If $u(\Diamond \lambda) \leq \alpha$ then we obtain by \#4, that  $\nu(\lambda) \leq \nu(\Diamond \lambda) < 1$. Then by Lemma \ref{normalization} we have $w(\lambda) \leq w(\Diamond \lambda) = v(\Diamond \lambda)$ and thus $w(\lambda) \Rightarrow v(\Diamond \lambda) = 1$.  \hfill
\end{description}
\end{proof2}
\vspace{0.2cm}

\noindent \textbf{Claim 2 from Lemma  \ref{equation-joint}.  }\emph{If $u(\Diamond \psi )=\alpha >0$ then, for any small enough $\varepsilon >0,$ there exists a valuation $w'\in W$ such that $\min(w'(\psi ), \pi^c(w')) \geq \alpha -\varepsilon$}.

\begin{proof2}
Assume $u(\Diamond \psi )=\alpha >0$\footnote{In this context, $\alpha$ 
plays the role of $\beta$ in Lemma \ref{normalization}}. Then, we take $\delta = \max \{ u(\lambda) < \alpha : \lambda \in \Delta_\varphi \}$ (note that the last set is not empty because $\Diamond \gbot$ belongs to it) 
and define $\Gamma'_{\psi,u}$ in a similar way that it was defined in the 
proof of Claim 1:

\begin{equation*}
\begin{array}{ll}
\Gamma'_{\psi,u}= & \{\chi : \chi \in X \mbox{ and } u(\chi) > \delta \} \\
& \cup \{\lambda \rightarrow \theta  : \lambda \in \Delta_\varphi \mbox{ and } u(\lambda) \leq u(\Box \theta) \} \\
& \cup \{(\theta \rightarrow \lambda)\rightarrow \lambda : \lambda \in \Delta_\varphi \mbox{ and }  u(\lambda)< u(\Box \theta)< 1 \} \\
& \cup \{\theta   \rightarrow \lambda : \lambda \in \Delta_\varphi \mbox{ 
and } u(\Diamond \theta) \leq u(\lambda) \} \\
& \cup \{(\lambda \rightarrow \theta)\rightarrow \theta : \lambda \in \Delta_\varphi \mbox{ and }  u(\Diamond \theta)< u(\lambda)< 1 \} \\
& \cup \{(\chi_1 \rightarrow \chi_2) \rightarrow \chi_2 : \chi_1 , \chi_2 
\in X \mbox{ and } u(\chi_2) < u(\chi_1) \leq \delta \} \\
\end{array}%
\end{equation*}%

If $u(\Diamond \psi )=\alpha >0$, then we let $U_{\psi,u}^{{}} = \psi 
\to \theta$ where $\theta \in \Delta_\varphi$ and $u(\theta) = \delta$. 
We claim that
\begin{equation*}
\Gamma'_{\psi,u} \not\vdash _{K45(\mathbf{G})} U_{\psi ,u} \ ,
\end{equation*}
otherwise there would exist $\theta_1, \ldots, \theta_n \in \Gamma'_{\psi,u}$ such that $\vdash _{K45(\mathbf{G})} (\theta_1\wedge \ldots \wedge \theta_n) \to U_{\psi ,u}$, and then we would have $\vdash _{K45(\mathbf{G})} \Box(\theta_1\wedge \ldots \wedge \theta_n)  \to \Box (\psi \to \theta)$, which would imply $\vdash _{K45(\mathbf{G})} (\Box\theta_1\wedge \ldots \wedge \Box\theta_n)  \to (\Diamond \psi \to \Box \theta)$.
In that case, evaluating with $u$ it would yield $ 1= u(\Box \theta_1\wedge \ldots \wedge \Box\theta_n ) \Rightarrow u(\Diamond \psi \to \theta )$, contradiction, since $u(\Box\theta_1\wedge \ldots \wedge \Box\theta_n) > \delta$  and $u(\Diamond \psi \to \theta) = \delta$ (because by definition $u(\theta)< u(\Diamond \psi)$).

Therefore, there is an evaluation $\nu'$ such that $\nu'(ThK45(\mathbf{G}) )=\nu'(\Gamma'_{\psi,u})=1$, $\nu'(\Diamond \psi)=1$ and $\nu'(\psi \to \theta)<1$. Hence, Condition $\bf a$, in Lemma \ref{normalization} 
is satisfied. Condition $\bf b $ is also satisfied because by definition of $\Gamma'_{\psi,u}$ we know that $\nu'(\chi) = 1$ for every $\chi \in 
X$ such that $u(\chi) > \delta$. In addition, we can verify Condition $\bf c$ by considering Property $\#6$ of Claim 1 with $\delta$ instead of $\alpha$.

Therefore, for any $\varepsilon >0$ with $\delta + \varepsilon < \alpha$ there exists a valuation $w' \in W$ such that $w'(\psi)\geq \delta + \varepsilon$. Note that in this case $\delta < \alpha=1$. That means we are 
always able to find an appropriate valuation $w'$ such that $w'(\psi)$ is 
arbitrarily close to $\alpha$. 

It remains to show that $\pi^c(w') \geq \alpha$. Indeed, by construction, 
it holds that $\forall \theta \in \Delta_\varphi \mbox{ and } u(\Box \theta) \leq \delta : u(\Box \theta) \leq w'(\theta )\leq u(\Diamond \theta)$, and hence $\min(u(\Box \theta) \Rightarrow w'(\theta ) , w'(\theta )\Rightarrow u(\Diamond \theta )) \geq \alpha$. This finishes the proof.
\end{proof2}

\end{document}